\documentclass[10pt,a4paper]{amsart}
\usepackage{fullpage} 
\usepackage{amsmath}
\usepackage{amsthm}
\usepackage{amssymb}
\usepackage{mathtools}
\usepackage[only,mapsfrom]{stmaryrd}
\usepackage{graphicx}
\usepackage{cancel} 
\usepackage{color} 
\usepackage[all]{xy} 
\usepackage{tikz}
\usetikzlibrary{arrows,calc}
\usetikzlibrary{cd}
\usepackage{nicefrac}
\usepackage{paralist}
\usepackage{enumitem}
\usepackage[colorlinks,citecolor=blue,linkcolor=blue,urlcolor=blue,filecolor=blue,breaklinks]{hyperref}
\usepackage{verbatim}  %includes comment environment

\numberwithin{equation}{section}
\numberwithin{figure}{section}

{\begin{compactitem}

}%
{\end{compactitem}}

\newtheorem{thm}{Theorem}[section]

\newtheorem{lem}[thm]{Lemma}
\newtheorem{prop}[thm]{Proposition}

\newtheorem{ques}[thm]{Question}

\newtheorem{obs}[thm]{Observation}
\newtheorem*{introthm}{Theorem}
\newtheorem*{thm*}{Theorem}
\newtheorem*{conj*}{Conjecture}
\newtheorem*{cor*}{Corollary}
\newtheorem*{ques*}{Question}

\newtheorem*{namedthm}{\namedthmname}

\theoremstyle{definition}

\newtheorem{defn}[thm]{Definition}

\newtheorem*{rem*}{Remark}

\newcommand{\cC}{\mathcal{C}}

\newcommand{\cM}{\mathcal{M}}

\newcommand{\cR}{\mathcal{R}}
\newcommand{\cS}{\mathcal{S}}
\newcommand{\cT}{\mathcal{T}}

\newcommand{\bN}{\mathbb{N}}

\newcommand{\bR}{\mathbb{R}}

\newcommand{\bZ}{\mathbb{Z}}

\newcommand{\End}{{\mathrm{End}}}

\DeclareMathOperator{\Map}{Map}

\DeclareMathOperator{\Homeo}{Homeo}

\DeclareMathOperator{\genus}{genus}

\newcommand{\incl}[3][right]%
{%
\draw[<-,>=#1 hook] #2 to ($ #2!0.5!#3 $);
\draw[->] ($ #2!0.5!#3 $) to #3;%
}

\newcommand{\abs}[1]{\left \vert {#1} \right \vert}
\newcommand{\s}[1]{\left\{ {#1} \right\}}
\newcommand{\parens}[1]{\left( {#1} \right)}

\title{The number of ends of big mapping class groups}

\author{Josiah Oh}
\address{Shanghai Center for Mathematical Sciences, Jiangwan Campus, Fudan University, No.2005 Songhu Road, Shanghai, 200438, P.R. China}
\email{josiahoh@gmail.com}

\author{Yulan Qing}
\address{Department of Mathematics, University of Tennessee at Knoxville, 1403 Circle Drive, Knoxville, Tennessee}
\email{yqing@utk.edu}

\author{Xiaolei Wu}
\address{Shanghai Center for Mathematical Sciences, Jiangwan Campus, Fudan University, No.2005 Songhu Road, Shanghai, 200438, P.R. China}
\email{xiaoleiwu@fudan.edu.cn}

\date{Oct. 2025}

\begin{document}

\begin{abstract}
We analyze the number of ends of the mapping class group of a stable avenue surface. We prove that the mapping class group is one-ended whenever the stable avenue surface has at least one end of discrete type. Our method is to show that the associated translatable curve graph, which is quasi-isometric to the mapping class group, is one-ended.
\end{abstract}
\maketitle

\section{Introduction}

A central theme in geometric group theory is the \emph{large-scale geometry} of groups. By this we mean the metric structure that is preserved under quasi-isometries—maps which preserve distances up to a uniformly bounded multiplicative and additive error (Section \ref{sec:quasi-isometry}). In a 1983 ICM address~\cite{gromov-icm}, Gromov advocated the study of finitely generated groups as geometric objects and their classification up to quasi-isometry. Among the earliest and most classical quasi-isometric invariants is the \emph{number of ends}. The number of ends of a simplicial graph is the number of topologically distinct ways to escape to infinity, and the number of ends of a finitely generated group is defined to be the number of ends of its Cayley graph with respect to some (any) finite generating set (Section \ref{sec:ends of graphs}). Freudenthal \cite{Freudenthal} and Hopf \cite{Hopf} independently proved that the number of ends of a finitely generated infinite group is always either one, two, or infinite, and the group has exactly two ends if and only if it is virtually infinite cyclic. A celebrated theorem of Stallings~\cite{Stallings68, MR415622} provides a complete characterization of the case of infinitely many ends. It states that a finitely generated group has more than one end if and only if it splits over a finite group, or equivalently, it admits an action on a simplicial tree with finite edge-stabilizers and without edge-inversions or a global fixed vertex. More generally, countably infinite groups, even when not finitely generated, enjoy the same trichotomy on their number of ends \cite{yves}. 

It is a classical fact that mapping class groups of finite-type surfaces, except for some cases when the surface has low complexity, are all one-ended (all but a few even have geodesic divergence bounded above by a quadratic function \cite{DuchinRafi}). So the goal of this paper is to investigate the number of ends of \emph{big mapping class groups}, i.e., mapping class groups of infinite-type surfaces (Section \ref{sec:surfaces and mapping class groups}). Interest in big mapping class groups grew significantly after Bavard \cite{Bavard16} proved that the ray graph is Gromov hyperbolic, analogous to the famous theorem of Masur--Minsky \cite{MasurMinsky} that the curve graph is Gromov hyperbolic. Since then, much effort has been made to understand big mapping class groups and to compare them with the better-studied mapping class groups of finite-type surfaces (see for example, \cite{Long, Wu1, Israel}). 

Geometric group theory typically concerns finitely (or compactly) generated groups because it is well-known that the word metrics associated to any two finite (or compact) generating sets are quasi-isometric to each other. Big mapping class groups are not finitely or countably infinitely generated, nor are they compactly generated. So a priori, they do not have a well-defined geometry. However, Rosendal \cite{Rosendal} developed ideas which allow for the study of the large-scale geometry of topological groups which may not be locally compact, let alone finitely generated. Specifically, Rosendal proved that for groups which admit a \textit{coarsely bounded} (CB) generating set (Section \ref{sec:CB}), it is also the case that the word metrics associated to any two CB generating sets are quasi-isometric to each other. We remark that a subset of a locally compact space is CB if and only if it is compact, so this notion does indeed generalize older ones. Recent work of Mann--Rafi \cite{MannRafi} applied this idea to classify (many of) the big mapping class groups which are CB generated, and therefore have a well-defined quasi-isometry type. Using their tools and classification, several results have shown that these groups admit a rich geometric structure \cite{GRV, HQR, RoMo25}. See also \cite{Hill25} for the case of pure mapping class groups.

Within this framework, the number of ends of a CB generated group may be defined to be the number of ends of its Cayley graph with respect to some (any) CB generating set. Then it is natural to ask: \emph{How many ends do CB generated big mapping class groups have?} A theorem in \cite{OhPengitore22} states that the number of ends of a connected, unbounded, (coarsely) transitive graph is one, two, or infinite, generalizing the theorem of Freudenthal--Hopf. The theorem statement includes an assumption that the graph is locally finite, but in fact, this assumption is unnecessary as it is never used in the proof. The Cayley graph of a CB generated mapping class group is connected and transitive. So even though it is locally infinite (in general, the degree of each vertex is uncountable), we may conclude the following.

\begin{obs}
The number of ends of a CB generated mapping class group is zero, one, two, or infinite. The number of ends is zero if and only if the group is CB.
\end{obs}

A surface $\Sigma$ of infinite type falls into one of two classes depending on whether or not it contains a finite-type \textit{non-displaceable subsurface}---a subsurface which intersects its image under every homeomorphism of $\Sigma$. According to recent work, the large-scale geometry of the big mapping class group $\Map(\Sigma)$ often depends on whether or not $\Sigma$ contains a non-displaceable subsurface of finite type. For example, if $\Sigma$ does contain a non-displaceable subsurface of finite type, then $\Map(\Sigma)$ is not CB \cite{MannRafi} and it admits a continuous and nonelementary action by isometries on a Gromov hyperbolic space \cite{HQR} (the authors also prove that under an extra condition, $\Map(\Sigma)$ admits such an action only if $\Sigma$ contains a non-displaceable subsurface). On the other hand, if $\Sigma$ does not contain a non-displaceable subsurface of finite type and $\Map(\Sigma)$ is not CB, then $\Map(\Sigma)$ has infinite asymptotic dimension \cite{GRV}.

In an effort to further classify the number of ends of big mapping class groups, we focus on those surfaces $\Sigma$ which do not contain a non-displaceable subsurface of finite type. Since we are attempting to analyze the large-scale geometry of $\Map(\Sigma)$, we are only interested in the surfaces for which $\Map(\Sigma)$ is CB generated, but not CB itself. A broad class of surfaces which satisfy these conditions is the class of \textit{avenue surfaces} (Section \ref{sec:avenue surfaces}). Avenue surfaces were defined in \cite{HQR}, and the only additional condition imposed on them is a weak topological condition called \textit{tameness} on the space of ends of $\Sigma$. Tameness is rather technical and lengthy to define, so we choose to work within the slightly more restrictive, but much more easily defined, class of \textit{stable} surfaces (Section \ref{sec:surface ends}). Stable surfaces form a substantial subclass of tame surfaces, and it has become somewhat standard to work with them. For example, several other works in the literature \cite{Bar-NatanVerberne,GRV,BDR} similarly restrict to stable surfaces in order to conveniently work within the framework developed in \cite{MannRafi}. That being said, we expect our results to also hold for surfaces with tame end space.

Every avenue surface has two distinct ends $e_+$ and $e_-$ which are maximal with respect to the partial pre-order $\preceq$ on $\End(\Sigma)$ defined in \cite{MannRafi}. It follows from \cite{Schaffer-Cohen} that for each end $x$ which is an immediate predecessor of $e_+$ (and $e_-$), the equivalence class of $x$ is either a discrete set or a totally disconnected perfect set, where the equivalence relation is defined by $\preceq$. In the first case, we say the end $x$ is of \textit{discrete type}. We can now state our main theorem.

\begin{introthm}
If $\Sigma$ is a stable avenue surface with a discrete-type end, then $\Map(\Sigma)$ is one-ended.
\end{introthm}

The most basic example of a surface covered by this theorem is the bi-infinite flute surface. This surface indeed has exactly two maximal ends, and all other ends, represented by punctures, form a single equivalence class homeomorphic to $\bZ$. On the other hand, one surface which is not covered by our theorem is the Jacob's ladder surface, which has exactly two ends, each accumulated by genus. It may seem at first that the Jacob's ladder surface should be no more difficult to handle than the bi-infinite flute surface. However, one significant difference is that in the bi-infinite flute surface, curves cannot meaningfully ``interact" with the punctures, since the punctures are not points on the surface, but in the Jacob's ladder surface, curves may ``interact" with the handles in a highly non-trivial way. We do not know whether our methods may be modified to deal with this additional layer of complexity, so we ask the following.

\begin{ques}
    Is the mapping class group of the Jacob's ladder surface one-ended? More generally, if $\Sigma$ is a stable avenue surface without any discrete-type ends, is $\Map(\Sigma)$ one-ended?
\end{ques}

A common strategy to study mapping class groups is to construct appropriate metric spaces on which the groups act with good geometric properties. For finite-type surfaces, these spaces are often graphs whose vertices are represented by curves and arcs on the surface, and the fundamental example is the curve graph defined in~\cite{Harvey}. Avenue surfaces fall in the broader class of \textit{translatable surfaces}, for which Schaffer-Cohen~\cite{Schaffer-Cohen} introduced the \emph{translatable curve graph}, whose vertices correspond to simple closed curves which separate the two maximal ends, and whose edges connect pairs of vertices whenever the corresponding curves bound a subsurface homeomorphic to one of a pre-chosen finite collection of subsurfaces (Section \ref{sec:translatable curve graph}). The translatable curve graph is not only an analogue of the curve graph in the setting of translatable surfaces; it is even quasi-isometric to the associated mapping class group \cite{Schaffer-Cohen}. So every property which is invariant under quasi-isometry holds for the mapping class group precisely when it holds for the translatable curve graph. Therefore, we directly analyze this graph in order to determine the number of ends of the mapping class group. To this end, we devise an operation called \textit{lassoing} (Section \ref{sec:lasso and flux}) that takes a vertex in the graph to an adjacent vertex and, more generally, takes a (well-behaved) path to an adjacent path. Furthermore, we find some lower bounds on the graph distance, called \textit{flux} (Section \ref{sec:lasso and flux}) and \textit{Hamming distance} (Section \ref{sec:hamming distance}), that are crucial to proving the graph is one-ended. Section~\ref{sec3} presents a proof of the theorem in the case of the bi-infinite flute, capturing many of the key ideas, and the theorem is proved in full in Section~\ref{sec4}. 

We finish with additional questions that arose during our investigation.

\subsection*{Open Questions}

This paper concerns only those surfaces which do not contain a non-displaceable subsurface of finite type. This, of course, raises the following question.

\begin{ques}
    Can one classify the number of ends of $\Map(\Sigma)$ when $\Sigma$ is an infinite-type surface which contains a non-displaceable subsurface of finite type?
\end{ques}

In particular, this question is open for the particular surface studied in \cite{Bavard16} that initiated much of the recent interest in big mapping class groups.

\begin{ques}
Let $\Sigma = \mathbb{R}^2 \setminus \mathcal{C}$ be the plane minus a Cantor set. How many ends does $\Map(\Sigma)$ have?
\end{ques}

Schaffer-Cohen proved that $\Map(\Sigma)$ is quasi-isometric to the loop graph \cite[Theorem 6.5]{Schaffer-Cohen}, which, in turn, is quasi-isometric to the ray graph~\cite[Proposition~3.11]{Bavard16}. Bavard--Walker gave a characterization of the Gromov boundary of the ray graph in~\cite{BavardWalker18}. However, their results do not immediately determine the number of ends of the ray graph.

A question that approaches the subject from another angle is:

\begin{ques}
Does there exist a big mapping class group with more than one end?
\end{ques}

Recall that a finitely generated group is two-ended if and only if it contains an infinite cyclic subgroup of finite index. This motivates the next question.

\begin{ques}
    Does there exist a big mapping class group with exactly two ends? If one exists, what can be said about its algebraic structure?
\end{ques}

\subsection*{Acknowledgments}
We thank Kasra Rafi for suggesting the notion of flux to us, and we thank Anschel Schaffer-Cohen for helpful discussions. Xiaolei Wu is currently a member of LMNS and is supported by NSFC No.\,12326601.

\section{Preliminaries and set-up}

\subsection{Quasi-isometry}\label{sec:quasi-isometry}

Let $f \colon (X, d_X) \to (Y, d_Y)$ be a map on metric spaces. If $L \geq 1$ and $A \geq 0$, we say that $f$ is an $(L,A)$\textit{-quasi-isometric embedding} if for every $a,b \in X$,
$$
\frac{1}{L} d_X(a,b) - A \leq d_Y(f(a), f(b)) \leq L d_X(a,b) + A.
$$
An $(L,A)$-quasi-isometric embedding $f$ is an $(L,A)$-\textit{quasi-isometry} if there is an $(L',A')$-quasi-isometric embedding $g : Y \to X$ such that $d_X(g\circ f,\text{Id}_X) < \infty$ and $d_Y(f\circ g,\text{Id}_Y) < \infty$, and we call $g$ a \textit{quasi-inverse} of $f$. Equivalently, an $(L,A)$-quasi-isometric embedding $f$ is an $(L,A)$-quasi-isometry if it is \textit{coarsely surjective}, that is, if there is a $C \ge 0$ such that the image of $f$ is $C$-dense in $Y$. A map $f \colon X \to Y$ is a \textit{quasi-isometry} between $X$ and $Y$ if it is an $(L,A)$-quasi-isometry for some $L \ge 1, A \ge 0$. Two metric spaces $X$ and $Y$ are \textit{quasi-isometric} if there exists a quasi-isometry between them.

\subsection{Number of ends of a graph}\label{sec:ends of graphs}

A \textit{graph} $X$ is a pair of sets $(V,E)$ where $E$ is a set of subsets of $V$ containing exactly two elements. That is 
\[ E \subset \{e\mid e\subset V, |e| =2\}.\]
We call $V=V(X)$ the set of \textit{vertices} and $E=E(X)$ the set of \textit{edges}. Given an edge $\s{x,y}$, we call $x$ and $y$ the \textit{endpoints} of $\s{x,y}$ and we say that $x$ and $y$ are \textit{adjacent} or \textit{neighbors}. A \textit{graph isomorphism} between graphs $X$ and $Y$ is a bijection $f : V(X) \to V(Y)$ such that $x$ and $y$ are adjacent if and only if $f(x)$ and $f(y)$ are adjacent. A \textit{graph automorphism} of a graph $X$ is a graph isomorphism from $X$ to itself. A graph is \textit{vertex-transitive} (or \textit{transitive}) if its automorphism group acts transitively on its vertices. We emphasize that we do not assume our graphs are locally finite. In fact, every vertex in our graphs of interest will have uncountable degree.

A \textit{path} between two vertices $x$ and $y$ is a sequence $(v_0,v_1,\dots,v_n)$ of vertices such that $v_0 = x$, $v_n = y$, and consecutive vertices are connected by an edge, i.e. $\{v_i,v_{i+1}\} \in E$ for $i=0,\cdots,n-1$. A graph is \textit{connected} if any two vertices can be connected by a path. Any connected graph $X$ is also a metric space via the path metric which defines the distance between two vertices to be the length of a shortest path between them. If $S \subset V(X)$, then the \textit{subgraph of $X$ induced by $S$} is the graph whose vertex set is $S$ and whose edge set is the subset of edges in $E(X)$ that have both endpoints in $S$. We re-use the symbol $S$ to denote this induced subgraph, and we denote by both $X \setminus S$ and $S^\mathsf{c}$ the subgraph of $X$ induced by $V(X) \setminus S$.

For a connected graph $X$ and subset $S\subset V(X)$, denote by $U(X,S)$ (or $\pi_0^u(S^\mathsf{c})$) the set of unbounded connected components of the subgraph $X \setminus S$. Define the \textit{number of ends} of $X$, denoted by $e(X)$, by
\[
    e(X) = \sup\s{\abs{U(X,B)} : B \text{ is a bounded subset of } X}.
\]

The number of ends turns out to be invariant under quasi-isometry. This statement is a well-known fact under conditions such as local finiteness, or local compactness for more general spaces, and proofs may be readily found in the literature (for example, \cite[Lemma 9.5]{drutu-kapovich}). This is because finitely generated groups are usually the objects of interest, and their Cayley graphs are locally finite. To the authors' knowledge, analogous statements without assuming conditions such as local compactness do not exist in the literature, so we include a proof here for the sake of completeness. 

Given a vertex $x$ and radius $r\ge0$, the ball of radius $r$ centered at $x$ is
\[
    B(x,r) = \s{y \in V(X) : d(x,y) \le r}.
\]
Note that if $B_1$ and $B_2$ are bounded subsets of $V(X)$, and $B_1 \subset B_2$, then $\abs{U(X,B_1)} \le \abs{U(X,B_2)}$. Since balls are bounded, and bounded sets are contained in balls, we have the following lemma.

\begin{lem}
Let $x \in V(X)$ be any vertex. Then
\[
    e(X) = \sup_{r>0} \abs{U(X,B(x,r))} = \sup_{\substack{r>0 \\ y\in X}}\abs{U(X,B(y,r))}.
\]
\end{lem}

The next lemma will be used to show that the number of ends is invariant under quasi-isometry.

\begin{lem}\label{lem:qi-graph-end}
Let $X$ and $X'$ be connected graphs, and let $B=B(x,r)$ be a ball in $X$. If $f:X \to X'$ is a quasi-isometry, then there is an $R>0$ such that the ball $B'=B(f(x),R)$ in $Y$ satisfies $\abs{U(X',B')} \ge \abs{U(X,B)}$.
\end{lem}

\begin{proof}
Suppose $f:X\to X'$ and $g:X'\to X$ are $(L,A)$-quasi-isometries that are quasi-inverses of each other. Let $x'=f(x)$, and by possibly increasing $A$, we may assume without loss of generality that $g(x') = x$. Put $r' = Lr+L^2 + 2LA$ and $B' = B(x',r')$. The goal is to construct a surjection $U(X',B') \to U(X,B)$. First observe that whenever $y \notin B'$,
\begin{align}\label{g(y) is far from x}
    d(x,g(y)) =
    d(gf(x),g(y)) \ge
    L^{-1}d(f(x),y)-A >
    L^{-1}r'-A=
    r+L+A.
\end{align}
That is, $g$ maps the complement of $B'$ into the complement of $B(x,r+L+A)$. Let $C\in U(X',B')$, and pick any $y \in C$. Since $y \notin B'$, we know that $g(y) \notin B$. So $g(y)$ must be contained in some connected component $D$ of $X\setminus B$. We now show that $g(C) \subset D$. Let $z \in C$ be adjacent to $y$. Then $g(z) \notin B$ by the inequality \ref{g(y) is far from x}, and
\[
    d(g(y),g(z)) \le
    Ld(y,z)+A =
    L+A.
\]
Let $\gamma$ be a shortest path in $X$ between $g(y)$ and $g(z)$. If $g(z) \notin D$, then $g(z)$ is contained in a different component of $X\setminus B$. Hence $\gamma$ must intersect $B$ at some point, call it $w$, and we get
\[
    d(x,g(y)) \le
    d(x,w) + d(w,g(y)) \le
    r+d(g(z),g(y)) \le
    r+L+A,
\]
contrary to the inequality  \ref{g(y) is far from x}. So $g(z) \in D$, and since $C$ is connected, it follows that $g(C) \subset D$. Since $C$ is unbounded and $g$ is a quasi-isometry, $D$ must be unbounded. Therefore $D \in U(X,B)$ and we put $\Phi(C) = D$. Thus we have a well-defined map $\Phi : U(X',B') \to U(X,B)$ which we will show is surjective. 

Let $D \in U(X,B)$ be arbitrary. Like before, we may take a ball $B_1=B(x,R)$ of sufficiently large radius $R$ such that whenever $y \notin B_1$, we have $d(x',f(y)) > r' + L + A$. Then by similar reasoning used before, $f$ maps each element of $U(X,B_1)$ into an element of $U(X',B')$. Since $D$ is an unbounded component of $X \setminus B$, and $B_1$ is just a bounded neighborhood of $B$, there must be a $D' \in U(X,B_1)$ with $D' \subset D$. Then like before, $f(D') \subset C$ for some $C \in U(X',B')$. Since $f$ and $g$ are quasi-inverse, and $D'$ is unbounded, there is a point in $C$ that $g$ maps into $D' \subset D$. Hence $g(C) \subset D$, and therefore $\Phi(C) = D$. Thus $\Phi$ is surjective, and $\abs{U(X',B')} \ge \abs{U(X,B)}$.
\end{proof}

\begin{prop}\label{prop:number of ends is a qi invariant}
If $X$ and $X'$ are connected graphs that are quasi-isometric, then $e(X) = e(X')$.
\end{prop}

\begin{proof}
By Lemma \ref{lem:qi-graph-end}, for any ball $B$ in $X$ there is a ball $B'$ in $X'$ such that $\abs{U(X',B')} \ge \abs{U(X,B)}$. It follows that $e(X') \ge e(X)$. Replacing the roles of $X$ and $X'$ yields the reverse inequality.
\end{proof}

Lastly, we establish a practical lemma.

\begin{lem}\label{lem:one-ended criterion}
    Let $X$ be a connected, unbounded, transitive graph, and fix a vertex $o$. Let $f:\bN \to \bN$ be any function with $f(n) > n$. Then $e(X)=1$ if the following holds for each integer $R>0$: Any $x_1,x_2\in X$ with $d(x_1,o)=d(x_2,o) = f(R)$ can be connected by a path which is disjoint from $B(o,R)$.
\end{lem}

\begin{proof}
    We prove the contrapositive statement. If $e(X) > 1$, then there exists an integer $R>0$ such that $B = B(o,R)$ satisfies $\abs{U(X,B)} \ge 2$. Let $C_1$ and $C_2$ be different unbounded components of $X \setminus B$. Since these are unbounded, one may choose for $i=1,2$, a vertex $x_i \in C_i$ with $d(x_i,o) = f(R)$. Since $B$ separates $C_1$ and $C_2$, all paths between $x_1$ and $x_2$ must intersect $B$. 
\end{proof}

\subsection{Surfaces and the mapping class group}\label{sec:surfaces and mapping class groups}
A \emph{surface} $\Sigma$ is a connected, orientable, $2$-dimensional topological manifold without boundary. It is of \emph{finite type} if its fundamental group is finitely generated; otherwise it is of \textit{infinite type}. A \emph{subsurface} $S$ of $\Sigma$ is a closed, connected subset of $\Sigma$ which is itself a surface with boundary. Assume the boundary of a subsurface consists of a finite number of pairwise disjoint simple closed curves, none of which bound a disk or a punctured disk in $\Sigma$. 

Denote by $\Homeo^+(\Sigma)$ the group of orientation-preserving homeomorphisms of $\Sigma$, with the compact-open topology, and denote by $\Homeo_0(\Sigma)$ the connected component of the identity in $\Homeo^+(\Sigma)$. The \emph{mapping class group} of $\Sigma$ is defined as the group $\Map(\Sigma) = \Homeo^+(\Sigma)/\Homeo_0(\Sigma)$ of all isotopy classes of orientation-preserving homeomorphisms of $\Sigma$, and $\Map(\Sigma)$ is endowed with the quotient topology. When $\Sigma$ is of infinite type, $\Map(\Sigma)$ is a Polish group which is not locally compact.

\subsection{The space of ends of a surface}\label{sec:surface ends}
The \textit{space of ends} $\End(\Sigma)$ of a surface $\Sigma$ is the inverse limit of the system of components of complements of compact subsets of $\Sigma$. Intuitively, each end corresponds to a way of leaving every compact subset of $\Sigma$ (see \cite{Richards} for details). Denote by $\End^g(\Sigma) \subset \End(\Sigma)$ the subspace of ends which are accumulated by genus, and denote by $\genus(\Sigma)$ the (possibly infinite) genus of $\Sigma$. By a theorem of Richards \cite{Richards}, connected, orientable surfaces $\Sigma$ are classified up to homeomorphism by the triple $(\genus(\Sigma),\End(\Sigma),\End^g(\Sigma))$. For a subsurface $S\subset\Sigma$, the space of ends of $S$ is defined similarly and is denoted by $\End(S)$. The embedding of $S$ in $\Sigma$ gives a natural embedding of $\End(S)$ into $\End(\Sigma)$. 

%Every subsurface $S\subseteq\Sigma$ determines a finite partition $\Pi_S$ of the ends of $\End(\Sigma) \setminus \End(S)$, where each element of the partition is the space of ends of a connected component of $\Sigma \setminus S$. Given two subsets $X,Y\subseteq \End(\Sigma)$, we say that $S$ \emph{separates} $X$ and $Y$ if $X$ and $Y$ belong to distinct elements of the partition $\Pi_S$. 

Assume that $\Sigma$ is a surface of infinite type. One of the key tools used in \cite{MannRafi} was a partial pre-order $\preceq$ on $\End(\Sigma)$. By definition, $y \preceq x$ if, for every clopen neighborhood $U$ of $x$, there exists a clopen neighborhood $V$ of $y$ which is homeomorphic to a subset of $U$. This induces an equivalence relation on $\End(\Sigma)$, where $x$ and $y$ are \textit{equivalent} ends if $x \preceq y$ and $y \preceq x$, and we denote by $E(x)$ the equivalence class of $x$. Denote by $\cM(\End(\Sigma))$ the set of ends of $\Sigma$ which are maximal with respect to $\preceq$, that is, the set of ends $x \in \End(\Sigma)$ such that $x \preceq y$ implies $y \in E(x)$. The elements of $\cM(\End(\Sigma))$ are called \emph{maximal ends}.

For an end $x\in \End(\Sigma)$, a clopen neighborhood $U$ of $x$ is \emph{stable} if for any smaller neighborhood $U'\subset U$ of $x$, there is a homeomorphic copy of $U$ contained in $U'$. (see \cite[Proposition 3.2]{BDR} for equivalent definitions of a stable neighborhood). The surface $\Sigma$ is \textit{stable} if each of its ends has a stable neighborhood.

\subsection{Geometry of big mapping class groups} \label{sec:CB}

Let $G$ be a Polish group. A subset $ A \subset G$ is \textit{coarsely bounded}, abbreviated CB, if it has finite diameter with respect to every compatible left-invariant metric on G. A group is \textit{locally CB} if it has a CB neighborhood of the identity, and it is \textit{CB generated} if it has a CB generating set. It is a fact that a CB generated group is necessarily locally CB \cite[Theorem 1.2]{Rosendal}. From the point of view of large-scale geometry, CB sets in Polish groups are analogous to finite sets in discrete groups. In particular, a CB generated group admits a metric that is well-defined up to quasi-isometry.

\begin{thm}\label{thm:qi type is well-defined}\cite{Rosendal}. Let G be a CB generated Polish group. Then the identity map on $G$ is a quasi-isometry between the word metrics associated to any two symmetric, CB generating sets. 
\end{thm}

Thus, we can reasonably define the number of ends of CB generated Polish groups, and hence, CB generated big mapping class groups.

\begin{defn}
    Let $G$ be a CB generated Polish group, and fix a symmetric CB generating set $S$. The \textit{number of ends} of $G$ is the number of ends of its Cayley graph with respect to $S$.
\end{defn}

By Theorem \ref{thm:qi type is well-defined}, any other choice of $S$ yields a quasi-isometric Cayley graph which has the same number of ends by Proposition \ref{prop:number of ends is a qi invariant}.

\subsection{Avenue surfaces and translatable surfaces}\label{sec:avenue surfaces}

Next we recall the definition of an avenue surface and a translatable surface. These notions were introduced in \cite{HQR} and \cite{Schaffer-Cohen}, respectively, and we explain their precise relationship.

\begin{defn}
An avenue surface is a surface $\Sigma$ which does not contain any non-displaceable subsurface of finite type, whose end space is tame, and whose mapping class group $\Map(\Sigma)$ is CB-generated but not CB.
\end{defn}

The condition of tameness only requires that certain ends of $\Sigma$ admit a stable neighborhood. Stable surfaces, in which every end admits a stable neighborhood, automatically have tame end space. Since we will eventually only consider stable surfaces, we refer the interested reader to \cite{MannRafi} for a precise definition of tame and omit writing it here.

\begin{comment}
let $\Sigma$ be any avenue surface $\Sigma$, there is necessarily a mapping class $h$ that acts on the surface by \emph{translation}. That is, there exists a subsurface $S$ with two boundary components such that

\begin{itemize}
    \item $\{h^n \cdot S\}$ covers $\Sigma$,
    \item the cyclic subgroup $\langle h \rangle$ is coarsely embedded in $\Map(\Sigma)$.
    \item $\Sigma$ is a union of $\{h^n \cdot S\}$ where two copies of $S$ are either disjoint or identified along exactly one boundary component. Following the notation in \cite{Schaffer-Cohen}, we write $\Sigma \cong S^{\natural \bZ}$.
\end{itemize}
\end{comment}

\begin{defn}
    Given a surface $\Sigma$, an end $e$ of $\Sigma$, and a sequence $\alpha_n$ of simple closed curves on $\Sigma$, we write $\lim_{n\to\infty} \alpha_n = e$ if for each neighborhood $V$ of $e$, all but finitely many $\alpha_n$ are contained in $V$ (after isotopy). A homeomorphism $h$ of $\Sigma$ is a \textit{translation} if there are two distinct ends $e_+$ and $e_-$ of $\Sigma$ such that for any simple closed curve $\alpha$ on $\Sigma$, $\lim_{n\to\infty} h^n(\alpha) = e_+$ and $\lim_{n\to\infty} h^{-n}(\alpha) = e_-$. A surface $\Sigma$ is \textit{translatable} if it admits a translation.
\end{defn}

By \cite[Theorem 5.3]{Schaffer-Cohen}, if a surface $\Sigma$ has tame end space, and $\Map(\Sigma)$ is CB generated and not CB, then $\Sigma$ is an avenue surface if and only if it is translatable. In particular, every stable avenue surface is translatable.

In a translatable surface $\Sigma$, a simple closed curve $\alpha$ is \textit{separating} if $\Sigma \setminus \alpha$ is disconnected and $e_+$ and $e_-$ are ends of different components. (Note that this condition is stronger than the usual notion of separating which only requires $\Sigma \setminus \alpha$ to be disconnected.) If $\alpha$ is a separating curve, then the component of $\Sigma \setminus \alpha$ with $e_+$ (resp. $e_-$) as one of its ends is called the \textit{right} (resp. \textit{left}) \textit{side} of $\alpha$, and it is denoted by $\alpha_+$ (resp. $\alpha_-$). If $\beta$ is a separating curve that lies on the right side of $\alpha$, denote by $[\alpha,\beta]$ the subsurface $(\alpha_+ \cap \beta_-) \cup \alpha \cup \beta$ of $\Sigma$ with boundary components $\alpha$ and $\beta$. If $\beta$ lies on the left side of $\alpha$ instead, then by an abuse of notation, still use $[\alpha,\beta]$ to denote $[\beta,\alpha]$.

Next we recall a characterization of translatable surfaces from \cite{Schaffer-Cohen}. Let $S$ be any surface with exactly two boundary components, and denote by $S^{\natural \bZ}$ the surface obtained by arranging countably many copies of $S$ like $\bZ$ and then gluing consecutive copies of $S$ together along their boundary components in the natural way. By construction, this surface is translatable. Conversely, every translatable surface admits such a decomposition.

\begin{prop}\cite[Proposition 3.5]{Schaffer-Cohen}
If $\Sigma$ is a translatable surface with translation $h$, and $\alpha$ is any separating curve, then there is a subsurface $S = [\alpha,h^n(\alpha)]$ for some $n$ such that $\Sigma$ is homeomorphic to $S^{\natural \bZ}$.
\end{prop}

\subsection{The translatable curve graph}\label{sec:translatable curve graph}

We now recall the definition of the translatable curve graph from \cite{Schaffer-Cohen}. Let $\Sigma = S^{\natural \bZ}$ be a translatable surface and assume that $\Sigma$ is tame. By \cite[Lemma 4.4]{Schaffer-Cohen} and \cite[Proposition 4.7]{MannRafi}, $\End(S)$ has a finite, positive number of equivalence classes of maximal ends, and each such equivalence class intersects $\End(S)$ at either a finite set or a Cantor set. So we may fix, once and for all, a set $\cR = \s{f_1,\dots,f_N,c_1,\dots,c_M}$ of representatives of the equivalence classes of maximal ends of $S$ such that each $E(f_i) \cap \End(S)$ is finite and each $E(c_i) \cap \End(S)$ is a Cantor set. Then in the whole end space $\End(\Sigma)$, since $\Sigma = S^{\natural \bZ}$, each $E(f_i)$ is homeomorphic to $\bZ$ and each $E(c_i)$ is homeomorphic to a Cantor set with two points removed. We say an end of $\Sigma$ is of \textit{discrete type} if it is equivalent to some $f_i$. Observe that $\Sigma$ has at least one end of discrete type if and only if $N > 0$.

Pick mutually disjoint stable neighborhoods $V_{f_i}$ and $V_{c_i}$ of the elements of $\s{f_1,\dots,f_N,c_1,\dots,c_M}$. The definition of stable implies that each $V_{f_i}$ contains exactly one discrete-type end. %By \cite[Lemma 4.5]{Schaffer-Cohen}, the $V_{f_i}$ and $V_{c_i}$ may be chosen so that $\End(S)$ is equal to the disjoint union $\bigsqcup_{i=1}^k V_{f_i} \sqcup \bigsqcup_{i=1}^{k'} V_{c_i}$.
For each $i=1,\dots,N$, let $T_i$ be a subsurface of $S$ with two boundary circles and end space homeomorphic to $V_{f_i} \sqcup \bigsqcup_{j=1}^{M} V_{c_j}$. If $f_i$ or any of the $c_j$ is accumulated by genus, then $T_i$ necessarily has infinite genus; otherwise, we require $T_i$ to have genus 0. If $S$ has finite positive genus, then let $T_{N+1}$ be a subsurface with two boundary circles, genus 1, and end space homeomorphic to $\bigsqcup_{j=1}^{M} V_{c_j}$. Put $\cS = \s{T_1,\dots,T_N,T_{N+1}}$ (including $T_{N+1}$ only if it is defined, of course). The \textit{translatable curve graph} $\cT\cC(\Sigma)$ of $\Sigma$ is the graph whose vertices are the isotopy classes of separating curves on $\Sigma$, with an edge between two curves $\alpha$ and $\beta$ if they have disjoint representatives and $[\alpha,\beta]$ is homeomorphic to a subsurface in $\cS$. By \cite[Lemma 4.7]{Schaffer-Cohen}, $\cT\cC(\Sigma)$ is connected, and we denote its path metric by $d$. Observe that $\Map(\Sigma)$ naturally acts by isometries on $\cT\cC(\Sigma)$. Furthermore, this action is transitive \cite[Lemma 3.6]{Schaffer-Cohen}.

There is an edge case: if $N=0$ and the genus of $S$ is either $0$ or infinite, then the above construction gives $\cS = \emptyset$. This is not desirable, so in this case we put $\cS = \{S\}$. In some sense this choice is consistent with the construction above because $S$ is indeed a subsurface with two boundary circles and end space homeomorphic to $\bigsqcup V_{c_i}$. More importantly, in this case $\cT\cC(\Sigma)$ has diameter 2 and $\Map(\Sigma)$ is CB, which means that both are quasi-isometric to a point (see \cite{Schaffer-Cohen} for details). If $\Sigma$ is a stable avenue surface, then $\Map(\Sigma)$ is not CB by assumption, and so $N=0$ implies that $S$ has finite positive genus.

Since every stable avenue surface $\Sigma$ is a translatable surface with tame end space, the following result allows us to use the translatable curve graph to understand the large-scale geometry of $\Map(\Sigma)$.

\begin{thm}\cite[Theorem 4.9]{Schaffer-Cohen}\label{thm:mapping class group is quasi-isometric to its translatable curve graph}
If $\Sigma$ is a stable avenue surface, then $\Map(\Sigma)$ is quasi-isometric to $\cT\cC(\Sigma)$.
\end{thm}

\section{Example: bi-infinite flute}\label{sec3}

In this section we illustrate the strategy of the proof of the main theorem with an example that captures many of the key ideas. Let $\Sigma$ be the \emph{bi-infinite flute}, i.e. the genus 0 surface with two maximal ends, $e_+$ and $e_-$, such that the other ends, represented as punctures, form a single equivalence class which accumulates to both $e_+$ and $e_-$. Equivalently, $\Sigma = S^{\natural \bZ}$, where $S$ is the annulus with one puncture. In this case, the translatable curve graph $\mathcal{TC}(\Sigma)$ is the graph whose vertices are (isotopy classes of) separating curves, with an edge between two curves $\alpha$ and $\beta$ if they have disjoint representatives and $[\alpha,\beta]$ is homeomorphic to an annulus with one puncture. From now on, denote the translatable curve graph simply by $\Gamma$. We first introduce two operations that induce maps on separating curves in $\Sigma$, or equivalently, vertices in $\Gamma$. These operations will be referred to as ``forgetting a puncture" and ``lassoing a puncture".

\subsection{Forgetting a puncture}

Identify $\Sigma$ with the space $X$ which is defined to be the strip $\bR \times [0,1]$ modulo the equivalence relation $(x,0) \sim (x,1)$, with marked points $(n,0)$ for each $n\in\bZ$. Suppose we erase the marking at a point $(n,0)$ and call the new space $X'$. As marked spaces, $X'$ is isomorphic to $X$, and one isomorphism is given by the homeomorphism $X' \to X$ which is the identity map on $(-\infty,n-1] \times [0,1]$, a horizontal contraction on $[n-1,n+1]\times[0,1]$, and the unit left shift on $[n+1,\infty) \times [0,1]$. Call this homeomorphism $\phi_n$ and denote by $\iota : X \to X'$ the map which forgets the marking at $(n,0)$. Then $\Phi_n = \phi_n \circ \iota : X \to X$ is a homeomorphism and it maps separating curves to separating curves. One can similarly define a homeomorphism $\psi_n : X' \to X$ using the identity map on $[n+1,\infty)\times [0,1]$ and the unit right shift on $(-\infty,n-1]\times[0,1]$. We freely allow $\Phi_n$ to mean either $\phi_n \circ \iota$ or $\psi_n \circ \iota$ depending on which is more appropriate for a given situation. 

\begin{defn}
    The operation of \textit{forgetting the puncture} at $(n,0)$ is the map $V(\Gamma) \to V(\Gamma)$ defined by sending each separating curve in $\Sigma \cong X$ to its image under $\Phi_n$. 
\end{defn}

In the next lemma, we observe that forgetting a puncture maps each pair of adjacent vertices in $\Gamma$ either to a pair of adjacent vertices or to the same vertex.

\begin{lem}\label{lem:forgetting a puncture preserves or collapses edges}
If $\alpha$ and $\beta$ are neighbors in $\Gamma$ and $\alpha'$ and $\beta'$ are their images after forgetting a puncture, then either $\alpha'=\beta'$ or $\alpha'$ and $\beta'$ are neighbors.
\end{lem}

\begin{proof}
Assume that $[\alpha,\beta]$ is an annulus with puncture $p$. If $p$ is the forgotten puncture, then $\alpha'$ and $\beta'$ are isotopic and therefore equal as vertices in $\Gamma$. Otherwise, $[\alpha',\beta']$ is still an annulus with puncture $p$, and therefore $\alpha'$ and $\beta'$ are neighbors in $\Gamma$.
\end{proof}

This implies that forgetting a puncture maps paths to (possibly shorter) paths.

\begin{lem}\label{forgetting a puncture shortens paths}
Forgetting a puncture sends the vertices on a given path to the vertices on a path which is not longer than the original path.
\end{lem}

\begin{proof}
    Proceed by induction on the length $n$ of the given path. The case $n=1$ follows from Lemma \ref{lem:forgetting a puncture preserves or collapses edges}. So assume the statement holds when $n=k-1$, and consider a path $(\alpha_0,\alpha_1,\dots,\alpha_k)$ of length $k$. For each $i=0,\dots,k$, denote by $\alpha_i'$ the image of $\alpha_i$ after forgetting a puncture. Since $(\alpha_1,\dots,\alpha_k)$ is a path of length $k-1$, the induction hypothesis implies that $(\alpha_1',\dots,\alpha_k')$ (after removing consecutive repeated vertices) is a path of length at most $k-1$. Moreover, Lemma \ref{lem:forgetting a puncture preserves or collapses edges} implies that either $\alpha_0' = \alpha_1'$ or $\alpha_0'$ and $\alpha_1'$ are neighbors. Hence, $(\alpha_0',\alpha_1',\dots,\alpha_k')$ (after removing consecutive repeated vertices) is a path of length at most $k$. 
\end{proof}

\subsection{Lassoing a puncture}

Let $\alpha$ be a separating curve, and let $\ell$ be an arc in $\Sigma$ which intersects $\alpha$ and has only one endpoint equal to a puncture, where the puncture is viewed as a marked point. Denote by $\ell'$ the subarc of $\ell$ obtained by traveling along $\ell$ from the marked endpoint until the first intersection point with $\alpha$. In other words, $\ell'$ is the component of $\ell \setminus \alpha$ containing the marked endpoint.

\begin{defn}
    Given $\alpha$ and $\ell$ as above, the \textit{lasso curve} defined by $\alpha$ and $\ell$, denoted by $\alpha(\ell)$, is the simple closed curve obtained by tracing along $\alpha$ and $\ell'$.  In this case, the arc $\ell$ is called a \textit{lasso}. See Figure \ref{fig:lasso}.
\end{defn}

\begin{figure}
    \centering
    \includegraphics[width=.8\linewidth]{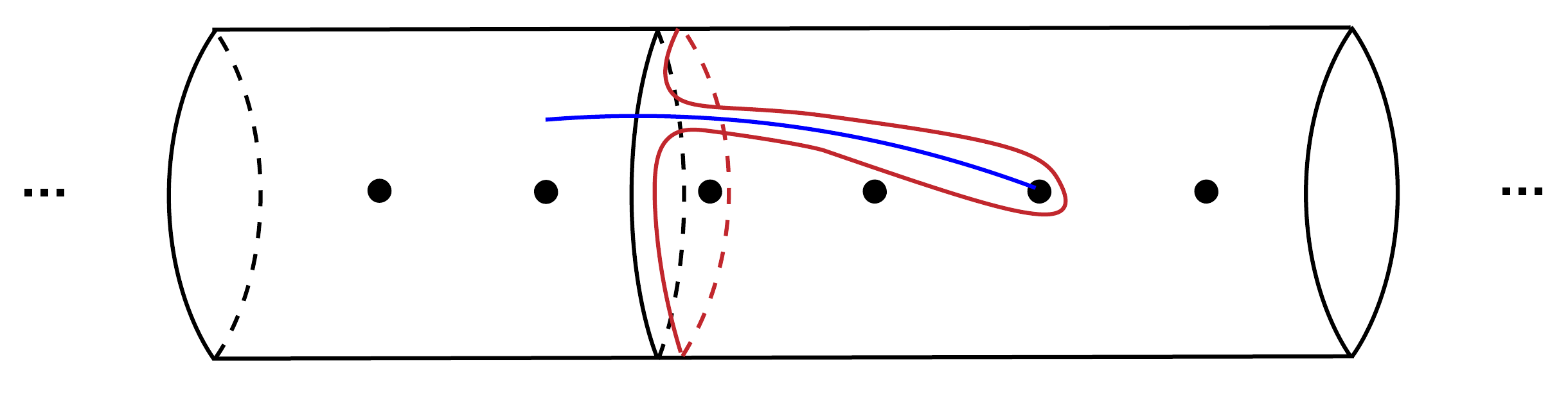}
    \caption{Lassoing a puncture: $\alpha$ is the black curve in the middle, $\ell$ is the blue arc, $\alpha(\ell)$ is the new red curve. }
    \label{fig:lasso}
\end{figure}

To give a more precise definition, consider a small tubular neighborhood of $\alpha \cup \ell'$. This subsurface has two boundary components, one of which is isotopic to $\alpha$. Then $\alpha(\ell)$ is equal to the other boundary component. The terminology is intended to evoke the image of the curve $\alpha$ using the arc $\ell$ as a lasso to wrangle the puncture $p$.

By construction, $\alpha(\ell)$ is a separating curve, which is adjacent to $\alpha$ in $\Gamma$. In fact, every neighbor of $\alpha$ can be realized as a lasso curve. Indeed, if $\alpha$ and $\beta$ are neighbors, then $[\alpha,\beta]$ is an annulus with one puncture $p$. So let $\ell$ be an arc in $[\alpha,\beta]$ from $\alpha$ to $p$, and note that $\alpha(\ell)$ is isotopic to $\beta$ because $[\alpha(\ell),\beta]$ is an annulus.

\begin{defn}
    Given a set of curves $\alpha_i$ in $\Sigma$, a connected subsurface $K$ of $\Sigma$ is a \textit{carrier} of the $\alpha_i$ if $K$ contains each $\alpha_i$ and has exactly two boundary components, each of which is a separating curve in $\Sigma$. 
\end{defn}

For example, every finite set of closed curves admits a carrier which is a ``finite flute", that is, a union of finitely many consecutive copies of $S$ in $\Sigma = S^{\natural \bZ}$. The next lemma says that given two curves $\alpha$ and $\beta$, taking a step in $\Gamma$ from $\alpha$ by lassoing a distant puncture is not a step towards $\beta$.

\begin{lem}\label{lem:lassoing a distant puncture moves away from the origin}
Let $\alpha$ and $\beta$ be distinct curves in a carrier $K$. Let $\ell$ be a lasso from $\alpha$ to a puncture $p\notin K$. Then in $\Gamma$, the vertex $\alpha(\ell)$ does not lie on any geodesic path between $\alpha$ and $\beta$. In particular, $d(\alpha(\ell),\beta) \ge d(\alpha,\beta)$. 
\end{lem}

\begin{proof}
Any path from $\alpha$ to $\beta$ which contains $\alpha(\ell)$ can be made shorter by replacing the initial subpath from $\alpha$ to $\alpha(\ell)$ with just the edge between $\alpha$ and $\alpha(\ell)$. So let $P$ be any path in $\Gamma$ from $\alpha$ to $\beta$ which takes its first step at $\alpha(\ell)$. We claim that $P$ is not a geodesic between $\alpha$ and $\beta$. By Lemma \ref{forgetting a puncture shortens paths}, forgetting the puncture $p$ transforms $P$ into another path $P'$ which is at most as long as $P$. Since $\alpha,\beta \subset K$ and $p\notin K$, forgetting $p$ maps $\alpha$ to $\alpha$ and $\beta$ to $\beta$. So $P$ is still a path from $\alpha$ to $\beta$. On the other hand, forgetting $p$ maps $\alpha(\ell)$ to $\alpha$. So $P'$ is strictly shorter than $P$.

If $d(\alpha(\ell),\beta) < d(\alpha,\beta)$, then $d(\alpha(\ell),\beta)=d(\alpha,\beta)-1$ because $\alpha$ and $\alpha(\ell)$ are neighbors. In this case, a geodesic between $\alpha(\ell)$ and $\beta$ together with the edge between $\alpha$ and $\alpha(\ell)$ would give a geodesic path between $\alpha$ and $\beta$ containing $\alpha(\ell)$. It was shown that no such geodesic exists; thus $d(\alpha(\ell),\beta) \ge d(\alpha,\beta)$.
\end{proof}

\subsection{Lasso paths}

Given a single curve $\alpha$ and a lasso $\ell$, we defined the lasso curve $\alpha(\ell)$. We now perform a similar construction given a set of mutually disjoint curves.

\begin{defn}\label{def:straight}
    A sequence of separating curves $\alpha_0,\dots,\alpha_n$ is \textit{straight} if $\alpha_i$ lies on the right side of $\alpha_{i-1}$ for each $i=1,\dots,n$. It is a \textit{straight path} if $(\alpha_0,\dots,\alpha_n)$ is a path in $\Gamma$.
\end{defn}

Let $\alpha_0,\dots,\alpha_n$ be a straight sequence of separating curves, and fix a carrier $K$. Let $\ell$ be an arc from $K_-$ to a puncture $p \in K_+$ such that $\ell$ intersects each $\alpha_i$ exactly once. Such an arc may be inductively constructed since $[\alpha_{i-1},\alpha_i]$ is connected and its boundary components are $\alpha_{i-1}$ and $\alpha_i$. Construct the lasso curves $\alpha_i(\ell)$ one by one, in order of decreasing index starting from $i=n$, so that $\alpha_{i-1}(\ell)$ is disjoint from $\alpha_i(\ell)$. This can always be done by drawing the ``lasso" portion of $\alpha_{i-1}(\ell)$ in a small enough neighborhood of $\ell$. If instead, $\ell$ is an arc from $K_+$ to a puncture $p \in K_-$ which intersects each $\alpha_i$ exactly once, then disjoint lasso curves $\alpha_i(\ell)$ may similarly be drawn. In both constructions, $\alpha_i(\ell)$ lies on the right side of $\alpha_{i-1}(\ell)$ for each $i=1,\dots,n$. So the following lemma is true.

\begin{lem}
    A straight sequence of separating curves remains straight after lassoing as above.
\end{lem}

Moreover, lassoing a straight path again yields a straight path.

\begin{lem}
    If $(\alpha_0,\dots,\alpha_n)$ is a straight path in $\Gamma$ and $\ell$ is a lasso as above, then $(\alpha_0(\ell),\dots,\alpha_n(\ell))$ is also a straight path in $\Gamma$.
\end{lem}

\begin{proof}
    Since $(\alpha_0,\dots,\alpha_n)$ is a path, each $[\alpha_{i-1},\alpha_i]$ is an annulus with one puncture. By the construction of the lasso curves $\alpha_i(\ell)$ and $\alpha_{i-1}(\ell)$ above, $[\alpha_{i-1}(\ell),\alpha_i(\ell)]$ is also an annulus with the same puncture as $[\alpha_{i-1},\alpha_i]$. So $\alpha_{i-1}(\ell)$ and $\alpha_i(\ell)$ are neighbors in $\Gamma$.
\end{proof}

\begin{defn}
    Given a straight path $(\alpha_0,\dots,\alpha_n)$ and a lasso $\ell$ as above, the \textit{lasso path} defined by the $\alpha_i$ and $\ell$ is the straight path $(\alpha_0(\ell),\dots,\alpha_n(\ell))$.
\end{defn}

See Figure \ref{fig:lasso-seq} for an example of a straight path and a lasso path.

\subsection{Flux}

The last ingredient needed for this section is the notion of flux. This will provide a lower bound for the graph distance $d$ in $\Gamma$. For disjoint separating curves $\alpha$ and $\beta$, define $F_0(\alpha,\beta)$ to be the number of punctures in $[\alpha,\beta]$. Note that $F_0(\alpha,\beta)=0$ if and only if $[\alpha,\beta]$ is an annulus if and only if $\alpha$ and $\beta$ are isotopic. To define the flux between pairs of intersecting curves, the next two lemmas are needed.

\begin{lem}\label{lem:flux additivity}
    If $\alpha,\beta,\gamma$ is a straight sequence of separating curves, then $F_0(\alpha,\gamma) = F_0(\alpha,\beta) + F_0(\beta,\gamma)$.
\end{lem}

\begin{proof}
    By definition, $\beta$ lies on the right side of $\alpha$ and $\gamma$ lies on the right side of $\beta$. Then $[\alpha,\gamma]$ is equal to the gluing of $[\alpha,\beta]$ and $[\beta,\gamma]$ along $\gamma$, and the equation follows.
\end{proof}

For separating curves $\alpha$ and $\beta$, if $\gamma$ is a separating curve in either $\alpha_+ \cap \beta_+$ or $\alpha_-\cap\beta_-$, define
\[
    F(\alpha,\beta;\gamma) =
    \abs{F_0(\alpha,\gamma) - F_0(\beta,\gamma)}
\]
The next lemma shows that the value of $F(\alpha,\beta;\gamma)$ does not depend on the choice of $\gamma$.

\begin{lem}\label{lem:reference curve independence}
    Let $\alpha$ and $\beta$ be separating curves. If $\gamma$ and $\gamma'$ are any two separating curves, each of which lies in either $\alpha_+ \cap \beta_+$ or $\alpha_-\cap\beta_-$, then $F(\alpha,\beta;\gamma) = F(\alpha,\beta;\gamma')$.
\end{lem}

\begin{proof}
    Assume without loss of generality that $\gamma \subset \alpha_+\cap\beta_+$; the case $\gamma \subset \alpha_- \cap \beta_-$ is handled similarly. Then there are two cases corresponding to whether $\gamma'$ lies in $\alpha_+ \cap \beta_+$ or $\alpha_-\cap\beta_-$. First suppose that $\gamma'$ also lies in $\alpha_+\cap\beta_+$. Let $\gamma''$ be a separating curve in $\gamma_+ \cap \gamma'_+$. Then by Lemma \ref{lem:flux additivity},
    \begin{align*}
        F(\alpha,\beta;\gamma) = 
        \abs{F_0(\alpha,\gamma) - F_0(\beta,\gamma)} 
        &=
        \abs{F_0(\alpha,\gamma) + F_0(\gamma,\gamma'') - F_0(\beta,\gamma) - F_0(\gamma,\gamma'')} \\&=
        \abs{F_0(\alpha,\gamma'') - F_0(\beta,\gamma'')} \\&=
        F(\alpha,\beta;\gamma'').
    \end{align*}
    Replacing $\gamma$ with $\gamma'$ in the above calculation also yields $F(\alpha,\beta;\gamma') = F(\alpha,\beta;\gamma'')$. Hence, $F(\alpha,\beta;\gamma) = F(\alpha,\beta;\gamma')$. Now suppose that $\gamma'$ lies in $\alpha_-\cap\beta_-$. By Lemma \ref{lem:flux additivity},
    \[
        F_0(\gamma',\gamma) = F_0(\gamma',\alpha) + F_0(\alpha,\gamma)
        \qquad \text{and} \qquad 
        F_0(\gamma',\gamma) = F_0(\gamma',\beta) + F_0(\beta,\gamma)
    \]
    Hence,
    \begin{align*}
        F(\alpha,\beta;\gamma) = 
        \abs{F_0(\alpha,\gamma) - F_0(\beta,\gamma)} 
        &=
        \abs{F_0(\gamma',\gamma) - F_0(\gamma',\alpha) - F_0(\gamma',\gamma) + F_0(\gamma',\beta)} \\&=
        \abs{F_0(\gamma',\alpha) - F_0(\gamma',\beta)} \\&=
        F(\alpha,\beta;\gamma').
    \end{align*}
\end{proof}

This allows for the following definition.

\begin{defn}
    Let $\alpha$ and $\beta$ be separating curves. The \textit{flux} between $\alpha$ and $\beta$, denoted by $F(\alpha,\beta)$, is defined to be $F(\alpha,\beta;\gamma)$ for some (any) choice of separating curve $\gamma$ in $\alpha_+\cap\beta_+$ or $\alpha_-\cap\beta_-$.
\end{defn}

Indeed, this definition agrees with $F_0$ when $\alpha$ and $\beta$ are disjoint.

\begin{lem}\label{lem:flux=F_0}
    If $\alpha$ and $\beta$ are disjoint separating curves, then $F(\alpha,\beta) = F_0(\alpha,\beta)$.
\end{lem}

\begin{proof}
    Assume without loss of generality that $\beta$ lies on the right side of $\alpha$. Let $\gamma$ be a separating curve in $\alpha_+ \cap \beta_+ = \beta_+$. Then by Lemma \ref{lem:flux additivity},
    \[
        F(\alpha,\beta) =
        \abs{F_0(\alpha,\gamma) - F_0(\beta,\gamma)} =
        \abs{F_0(\alpha,\beta) + F_0(\beta,\gamma) - F_0(\beta,\gamma)} =
        F_0(\alpha,\beta).
    \]
\end{proof}

Observe that if $\alpha$ and $\beta$ are neighbors in $\Gamma$, then $[\alpha,\beta]$ has one puncture and so $F(\alpha,\beta)=F_0(\alpha,\beta) = 1$. The next lemma implies that $F$ is a pseudometric on $\Gamma$.

\begin{lem}\label{lem:flux has triangle inequality}
All separating curves $\alpha,\beta,\gamma$ satisfy $F(\alpha,\beta) \le F(\alpha,\gamma) + F(\gamma,\beta)$.
\end{lem}

\begin{proof}
Let $\eta$ be a separating curve in $\alpha_+ \cap \beta_+ \cap \gamma_+$. Then
\begin{align*}
    F(\alpha,\beta) =
    \abs{F_0(\alpha,\eta)-F_0(\beta,\eta)} &=
    \abs{F_0(\alpha,\eta) - F_0(\gamma,\eta) + F_0(\gamma,\eta) - F_0(\beta,\eta)} \\&\le
    \abs{F_0(\alpha,\eta)-F_0(\gamma,\eta)}+\abs{F_0(\gamma,\eta)-F_0(\beta,\eta)} \\&=
    F(\alpha,\gamma)+F(\gamma,\beta).
\end{align*}
\end{proof}

However, it can happen that $F(\alpha,\beta) = 0$ for distinct $\alpha,\beta$. In fact, the Hamming distance (Definition \ref{defn:hamming-dist}) can be used to find $\alpha$ and $\beta$ with $F(\alpha,\beta)=0$ such that $d(\alpha,\beta)$ is arbitrarily large.

The notion of flux is useful yet because it provides a lower bound on the graph distance $d$ in $\Gamma$.

\begin{lem}\label{lem:flux is a lower bound}
All separating curves $\alpha$ and $\beta$ satisfy $d(\alpha,\beta) \ge F(\alpha,\beta)$.
\end{lem}

\begin{proof}
Suppose $(\gamma_0=\alpha,\gamma_1,\dots,\gamma_n=\beta)$ is a geodesic path in $\Gamma$ between $\alpha$ and $\beta$. Since $F$ is a pseudometric  by Lemma \ref{lem:flux has triangle inequality},
\[
    F(\alpha,\beta) \le
    \sum_{i=0}^{n-1} F(\gamma_i,\gamma_{i+1}) =
    n =
    d(\alpha,\beta).
\]
\end{proof}

It can now be shown that straight paths are geodesic.

\begin{lem}\label{lem:straight paths are geodesic}
    If $\alpha_0,\dots,\alpha_n$ is a straight path, then $d(\alpha_0,\alpha_n) = n$.
\end{lem}

\begin{proof}
    Clearly, $d(\alpha_0,\alpha_n) \le n$. On the other hand, it follows from an inductive argument using Lemma \ref{lem:flux additivity} that $F(\alpha_0,\alpha_n) = n$. Then Lemma \ref{lem:flux is a lower bound} implies $d(\alpha_0,\alpha_n) \ge n$.
\end{proof}

In general, flux only gives a lower bound on the graph distance. For disjoint separating curves, however, the flux is equal to the graph distance.

\begin{lem}\label{lem:flux=distance (flute)}
    If $\alpha$ and $\beta$ are disjoint separating curves, then there exists a straight path from $\alpha$ to $\beta$ of length $F(\alpha,\beta)$. In particular, $d(\alpha,\beta) = F(\alpha,\beta)$.
\end{lem}

\begin{proof}
    Since $\alpha$ and $\beta$ are disjoint, $F(\alpha,\beta) = F_0(\alpha,\beta)$ is equal to the number of punctures in $[\alpha,\beta]$. Denote these punctures by $p_1,\dots,p_n$. Put $\alpha_0=\alpha$. For $i=1,\dots,n$, inductively define $\alpha_i$ to be the lasso curve $\alpha_{i-1}(\ell_i)$, where $\ell_i$ is an arc from $\alpha_{i-1}$ to $p_i$ contained in $[\alpha_{i-1},\beta]$. By construction, $[\alpha_{i-1},\alpha_i]$ is an annulus with one puncture, and so $d(\alpha_{i-1},\alpha_i) = 1$. Moreover, $\alpha_n$ is isotopic to $\beta$ because $[\alpha_n,\beta]$ is just an annulus. Hence, $\alpha_0,\dots,\alpha_n$ is a straight path from $\alpha$ to $\beta$ of length $F(\alpha,\beta)$. By Lemma \ref{lem:straight paths are geodesic}, $d(\alpha,\beta) = F(\alpha,\beta)$.
\end{proof}

Lastly, flux gives a notion of left and right for pairs of separating curves which are not disjoint.

\begin{defn}
    For separating curves $\alpha$ and $\beta$, we say that $\beta$ is \textit{flux-right of} $\alpha$ if $F_0(\alpha,\eta) \ge F_0(\beta,\eta)$ holds for every separating curve $\eta$ in $\alpha_+ \cap \beta_+$. We say that $\beta$ is \textit{flux-left of} $\alpha$ if the inequality holds for every separating curve $\eta$ in $\alpha_- \cap \beta_-$.
\end{defn}

We collect some basic facts about this notion. Lemma \ref{lem:flux additivity} is used throughout the proof of the next lemma without explicit mention.

\begin{lem}\label{lem:flux-right facts}
    Let $\alpha$ and $\beta$ be separating curves. The following statements are true.
    \begin{enumerate}
        \item If $\beta$ lies on the right side of $\alpha$, then $\beta$ is flux-right of $\alpha$ and $\alpha$ is flux-left of $\beta$.

        \item $\beta$ must be either flux-right or flux-left of $\alpha$.

        \item $\beta$ is flux-right of $\alpha$ if and only if $\alpha$ is flux-left of $\beta$.

        \item If $\gamma$ is a separating curve such that $\beta$ is flux-right of $\alpha$ and $\gamma$ is flux-right of $\beta$, then $\gamma$ is flux-right of $\alpha$.

        \item $\beta$ is both flux-right and flux-left of $\alpha$ if and only if $F(\alpha,\beta)=0$.

        \item $\beta$ is flux-right of $\alpha$ if and only if $F(\alpha,\gamma) > F(\alpha,\beta)$ holds for every separating curve $\gamma \subset \beta_+$ not isotopic to $\beta$, and $\beta$ is flux-left of $\alpha$ if and only if the inequality holds for every separating curve $\gamma \subset \beta_-$ not isotopic to $\beta$. 
    \end{enumerate}
\end{lem}

\begin{proof}
    (1) For any separating curve $\eta$ in $\alpha_+ \cap \beta_+ = \beta_+$,
    \[
        F_0(\alpha,\eta) = 
        F_0(\alpha,\beta) + F_0(\beta,\eta) \ge
        F_0(\beta,\eta).
    \]
    Hence, $\beta$ is flux-right of $\alpha$, and similarly, $\alpha$ is flux-left of $\beta$.

    (2) Suppose $\beta$ is not flux-left of $\alpha$, and let $\gamma$ be a separating curve in $\alpha_- \cap \beta_-$ such that $F_0(\alpha,\gamma) < F_0(\beta,\gamma)$. Let $\eta$ be any separating curve in $\alpha_+ \cap \beta_+$. Then
    \[
        F_0(\gamma,\alpha) + F_0(\alpha,\eta) =
        F_0(\gamma,\eta) =
        F_0(\gamma,\beta) + F_0(\beta,\eta).
    \]
    Since $F_0(\alpha,\gamma) < F_0(\beta,\gamma)$, it follows that $F_0(\alpha,\eta) > F_0(\beta,\eta)$. Hence, $\beta$ is flux-right of $\alpha$.

    (3) Suppose $\beta$ is flux-right of $\alpha$, and let $\eta$ be a separating curve in $\alpha_+ \cap \beta_+$. Then $F_0(\alpha,\eta) \ge F_0(\beta,\eta)$. Let $\gamma$ be any separating curve in $\alpha_- \cap \beta_-$. Then
    \[
        F_0(\gamma,\alpha) + F_0(\alpha,\eta) =
        F_0(\gamma,\eta) =
        F_0(\gamma,\beta) + F_0(\beta,\eta)
    \]
    Since $F_0(\alpha,\eta) \ge F_0(\beta,\eta)$, it follows that $F_0(\gamma,\alpha) \le F_0(\gamma,\beta)$. Hence, $\alpha$ is flux-left of $\beta$. The reverse implication is proved similarly.

    (4) Let $\eta$ be any separating curve in $\alpha_+ \cap \gamma_+$, and let $\eta'$ be a separating curve in $\beta_+ \cap \eta_+$. By the definition of flux-right, $F_0(\alpha,\eta') \ge F_0(\beta,\eta') \ge F_0(\gamma,\eta')$. Then
    \[
        F_0(\alpha,\eta) =
        F_0(\alpha,\eta') - F_0(\eta,\eta') \ge
        F_0(\gamma,\eta') - F_0(\eta,\eta') =
        F_0(\gamma,\eta).
    \]
    Hence, $\gamma$ is flux-right of $\alpha$.

    (5) Suppose $F(\alpha,\beta) = 0$, and let $\eta$ be any separating curve in $\alpha_+ \cap \beta_+$. Then 
    \[
        0 = F(\alpha,\beta) = \abs{F_0(\alpha,\eta) - F_0(\beta,\eta)}.
    \]
    Hence, $F_0(\alpha,\eta) = F_0(\beta,\eta)$. So $\beta$ is flux-right of $\alpha$, and similarly, $\beta$ is shown to be flux-left of $\alpha$.
    
    Conversely, suppose that $\beta$ is both flux-right and flux-left of $\alpha$. By statement (3), $\alpha$ is flux-right of $\beta$. Let $\eta$ be a separating curve in $\alpha_+ \cap \beta_+$. Then $F_0(\alpha,\eta) \ge F_0(\beta,\eta)$ and $F_0(\beta,\eta) \ge F_0(\alpha,\eta)$. Hence, 
    \[
        F(\alpha,\beta) = 
        \abs{F_0(\alpha,\eta) - F_0(\beta,\eta)} = 
        0.
    \]
    
    (6) Suppose $\beta$ is flux-right of $\alpha$. Let $\gamma$ be any separating curve in $\beta_+$ not isotopic to $\beta$, and let $\eta$ be a separating curve in $\alpha_+ \cap \gamma_+$. Then $F_0(\alpha,\eta) \ge F_0(\beta,\eta)$. Since $\gamma$ is not isotopic to $\beta$, $F_0(\beta,\gamma) > 0$. So,
    \[
        F_0(\beta,\eta) = 
        F_0(\beta,\gamma) + F_0(\gamma,\eta) >
        F_0(\gamma,\eta).
    \]
    Hence, 
    \[
        F(\alpha,\gamma) =
        F_0(\alpha,\eta) - F_0(\gamma,\eta) >
        F_0(\alpha,\eta) - F_0(\beta,\eta) =
        F(\alpha,\beta).
    \]
    
    Conversely, suppose there exists a separating curve $\eta$ in $\alpha_+ \cap \beta_+$ such that $F_0(\alpha,\eta) < F_0(\beta,\eta)$. Then $[\beta,\eta]$ contains at least one puncture $p$. Let $\ell$ be an arc in $[\beta,\eta]$ from $\beta$ to $p$, and put $\gamma = \beta(\ell)$. Then $\gamma$ is a separating curve in $\beta_+$ not isotopic to $\beta$, and $F_0(\beta,\eta) = F_0(\gamma,\eta) + 1$. Since $F_0(\alpha,\eta) < F_0(\beta,\eta)$,
    \[
        F(\alpha,\gamma) =
        \abs{F_0(\alpha,\eta) - F_0(\gamma,\eta)} =
        \abs{F_0(\alpha,\eta) - F_0(\beta,\eta) + 1} =
        F(\alpha,\beta) - 1.
    \]
    
    The second statement is proved similarly.
\end{proof}

\subsection{Main theorem (special case)}

Now we are ready to prove the main theorem in the case of the bi-infinite flute.

\begin{thm}\label{thm:flute case}
The mapping class group of the bi-infinite flute is one-ended.
\end{thm}

\begin{proof}
By Theorem \ref{thm:mapping class group is quasi-isometric to its translatable curve graph}, it suffices to show that the translatable curve graph $\Gamma$ is one-ended. Fix a vertex $o$ of $\Gamma$ and let $R>0$ be any integer. To prove that $\Gamma$ is one-ended, it suffices by Lemma \ref{lem:one-ended criterion} to show that for any $\alpha,\beta \in \Gamma$ with $d(\alpha,o) = d(\beta,o) = 3R$, there exists a path from $\alpha$ to $\beta$ which is disjoint from $B = B(o,R)$. 

Since $\Map(\Sigma)$ acts transitively on $\Gamma$, we may assume without loss of generality that $\beta$ is a boundary curve of some copy of $S$ in $\Sigma = S^{\natural \bZ}$. By Lemma \ref{lem:flux-right facts} (2), $\beta$ is either flux-right or flux-left of $o$. Assume without loss of generality that $\beta$ is flux-right of $o$. Let $h$ be the translation on $\Sigma$ such that $h(\beta)$ lies on the right side of $\beta$ and $[\beta,h(\beta)]$ is homeomorphic to $S$. Then Lemma \ref{lem:flux-right facts} (6) implies $F(o,h(\beta)) > F(o,\beta)$. Moreover, $h(\beta)$ is flux-right of $\beta$ by Lemma \ref{lem:flux-right facts} (1), and so Lemma \ref{lem:flux-right facts} (4) implies that $h(\beta)$ is flux-right of $o$. It follows from induction that $F(o,h^i(\beta)) \ge F(o,\beta) + i$ for all $i \in \bN$. Since $h$ is a translation, we may fix a large enough $k$ so that $h^k(\beta)$ lies on the right side of $\alpha$. Put $\beta' = h^k(\beta)$. Then $(\beta,h(\beta),\dots,h^k(\beta))$ is a path from $\beta$ to $\beta'$. If this path has length at most $R$, then it must be disjoint from $B$ because $d(\beta,o) = 3R$. Otherwise if the length is greater than $R$, then the first $R+1$ vertices on the path lie outside of $B$ for the same reason as before, and the remaining vertices also lie outside of $B$ because for all $R < i \le k$,
\[
    d(o,h^i(\beta)) \ge
    F(o,h^i(\beta)) \ge
    F(o,\beta) + i >
    R.
\]
So $\beta$ and $\beta'$ are connected by a path which is disjoint from $B$. If instead $\beta$ is flux-left of $o$, then a similar argument gives a path $(\beta,h^{-1}(\beta),\dots,h^{-k}(\beta))$ from $\beta$ to some $h^{-k}(\beta)$ which is disjoint from $\alpha$. It now remains to construct a path from $\alpha$ to $\beta'$ which disjoint from $B$. 

Since $\beta'$ lies on the right side of $\alpha$, Lemma \ref{lem:flux=distance (flute)} gives a straight path $(\gamma_0=\alpha,\gamma_1,\dots,\gamma_n=\beta')$ from $\alpha$ to $\beta'$. We now describe a procedure to be iterated several times. Fix a carrier $K$ of $o$ and the $\gamma_i$. Let $\ell$ be an arc from $K_-$ to a puncture $p \in K_+$ such that $\ell$ intersects each $\gamma_i$ exactly once, and construct the lasso path $\gamma_0(\ell),\dots,\gamma_n(\ell)$. See Figure \ref{fig:lasso-seq} for an example. As lasso curves, each $\gamma_i(\ell_i)$ is adjacent to $\gamma_i$ in $\Gamma$. In particular, $\alpha$ and $\alpha(\ell)$ are connected by an edge, and so are $\beta'$ and $\beta'(\ell)$. Since $K$ is a carrier containing $o,\alpha,\beta'$, Lemma \ref{lem:lassoing a distant puncture moves away from the origin} implies
\[
    d(\alpha(\ell),o) \ge d(\alpha,o) > R
    \qquad \text{and} \qquad
    d(\beta'(\ell),o) \ge d(\beta',o) > R.
\]
So $\alpha(\ell),\beta'(\ell) \notin B$. Since the lasso path is again a straight path, the procedure described above may be iterated indefinitely. Indeed, after fixing a carrier $K'$ of $o$ and the $\gamma_i(\ell)$, a puncture $p' \in K'_+$, and an arc from $K'_-$ to $p'$ which intersects each $\gamma_i(\ell)$ exactly once, a new lasso path may again be constructed, the endpoints of which are not in $B$ and are adjacent to the endpoints of the previous path. Repeat the procedure for a total of $D+2R$ iterations, where $D = \max_i F(o,\gamma_i)$, and denote by $(\gamma_0',\dots,\gamma_n')$ the final path obtained after the last iteration. Now, observe that this process also yields a path from $\alpha$ to $\gamma_0'$ and a path from $\beta'$ to $\gamma_n'$ consisting of the endpoints of the lasso paths constructed in each iteration. Applying Lemma \ref{lem:lassoing a distant puncture moves away from the origin} at each iteration shows that both of these paths are disjoint from $B$. Once it is shown that $(\gamma_0',\dots,\gamma_n')$ is disjoint from $B$, the concatenation of the three paths then gives a path from $\alpha$ to $\beta'$ which is disjoint from $B$.

\begin{figure}
    \centering
    \includegraphics[width=.8\linewidth]{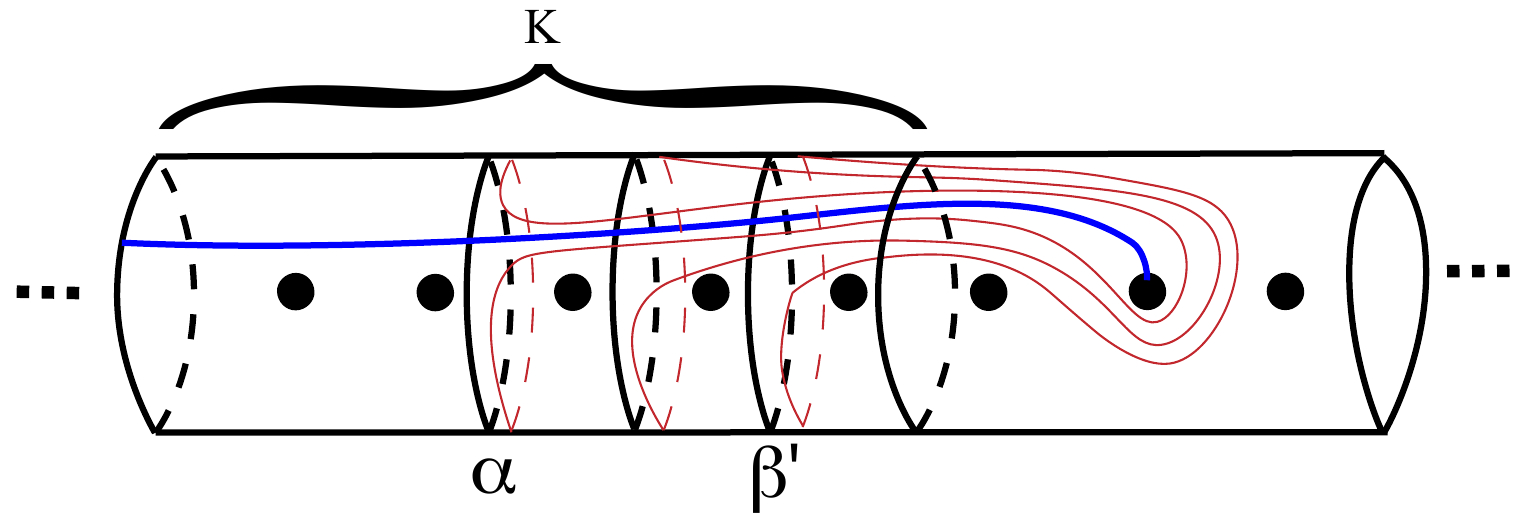}
    \caption{The lasso path constructed from a geodesic between $\alpha$ and $\beta'$ and an arc $\ell$ ending at a puncture outside a carrier $K$.}
    \label{fig:lasso-seq}
\end{figure}

Let $i \in \s{0,\dots,n}$. By construction, $\gamma_i$ and $\gamma_i'$ are disjoint and $[\gamma_i,\gamma'_i]$ is an annulus with $D+2R$ punctures. In particular, $F_0(\gamma_i,\gamma'_i) = D+2R$. By definition, $D \ge F(o,\gamma_i)$. Then by Lemma \ref{lem:flux has triangle inequality},
\[
    F(o,\gamma'_i) \ge
    F(\gamma_i,\gamma'_i) - F(o,\gamma_i) \ge
    (D+2R) - D =
    2R.
\]
So $d(o,\gamma'_i) \ge F(o,\gamma'_i) \ge 2R$, which implies $\gamma'_i \notin B$. Thus $(\gamma'_0,\dots,\gamma'_n)$ is disjoint from $B$.
\end{proof}

\section{Stable avenue surfaces}\label{sec4}

Now we consider the general setting when $\Sigma$ is a stable avenue surface. Again, we denote by $\Gamma$ the translatable curve graph $\cT\cC(\Sigma)$ of $\Sigma$. First we need to establish some notation. If a simple closed curve $\eta$ in $\Sigma$ does not separate $e_+$ and $e_-$, but $\Sigma \setminus \eta$ does have two components, then one of the components must have neither $e_+$ nor $e_-$ as an end. The union of this component with $\eta$ is a subsurface of $\Sigma$ with boundary $\eta$, and it is denoted by $C_\eta$. If a connected surface $T$ has at least one boundary circle, then denote by $\widehat T$ the surface obtained from $T$ by capping a disk onto a boundary circle. 

\subsection{General lassoing and flux}\label{sec:lasso and flux}

Let $\alpha$ be a separating curve in $\Sigma = S^{\natural \bZ}$, and let $V$ be a clopen subset of ends contained either in $\End(\alpha_+)$ or $\End(\alpha_-)$. By the proof of \cite[Lemma 4.6]{Schaffer-Cohen}, there exists a simple closed curve $\eta$ disjoint from $\alpha$ which bounds $V$. In particular, $\eta$ does not separate $e_+$ and $e_-$, and $\End(C_\eta) = V$. Moreover, if $S$ has finite positive genus, then for all $g \ge 0$, $\eta$ may be chosen so that $C_\eta$ has genus $g$. Let $\lambda$ be an arc which has one endpoint on $\eta$, is disjoint from $\eta$ otherwise, and intersects $\alpha$. Denote by $\lambda'$ the subarc of $\lambda$ obtained by traveling along $\lambda$ from its endpoint on $\eta$ until the first intersection point with $\alpha$. In other words, $\lambda'$ is the component of $\lambda \setminus \alpha$ containing the endpoint on $\eta$.

\begin{defn}
    Given $\alpha$, $\eta$, and $\lambda$ as above, the \textit{lasso curve} defined by $\alpha$, $\eta$, and $\lambda$, denoted by $\alpha(\eta,\lambda)$, is the simple closed curve obtained by tracing along $\alpha$, $\lambda$, and $\eta$. The pair $(\eta,\lambda)$ and the union $\eta \cup \lambda$ are both referred to as the \textit{lasso} of $\alpha(\eta,\lambda)$.
\end{defn}

Equivalently, $\alpha(\eta,\lambda)$ is the boundary curve of a small tubular neighborhood of $\alpha \cup \lambda' \cup \eta$ which is not isotopic to either $\alpha$ or $\eta$. Observe that by construction, the lasso curve $\alpha(\eta,\lambda)$ is a separating curve disjoint from $\alpha$, and the end space of $[\alpha,\alpha(\eta,\lambda)]$ is equal to $\End(C_\eta)$. Then the following lemma follows from the definition of the translatable curve graph.

\begin{lem}\label{lem:neighbor lasso curve}
    Let $\alpha$ be a separating curve and let $\alpha(\eta,\lambda)$ be a lasso curve defined by $\alpha$ and some $\eta$ and $\lambda$. Then $\alpha$ and $\alpha(\eta,\lambda)$ are neighbors in $\Gamma$ if either of the following conditions is met.
    \begin{itemize}
        \item  $C_\eta$ is homeomorphic to $\widehat T_i$ for some $1 \le i \le n$,
        \item $C_\eta$ is homeomorphic to $\widehat T_{n+1}$.
    \end{itemize}
\end{lem}

Let $(\alpha_0,\dots,\alpha_n)$ be a straight path in $\Gamma$. Let $\eta$ be a curve contained in either ${\alpha_n}_+$ or ${\alpha_0}_-$ which bounds a clopen subset of ends, and let $\lambda$ be an arc which has one endpoint on $\eta$, is disjoint from $\eta$ otherwise, and intersects each $\alpha_i$ exactly once. Then by the same construction and arguments used in the previous section, the \textit{lasso path} defined by the $\alpha_i$, $\eta$, and $\ell$ is the straight path $(\alpha_0(\eta,\lambda),\dots,\alpha_n(\eta,\lambda))$.

Next we extend the notion of flux for the more general surface $\Sigma$. For disjoint separating curves $\alpha$ and $\beta$, define $p_0(\alpha,\beta)$ to be the number of discrete-type ends of $[\alpha,\beta]$, and define 
\[
    g_0(\alpha,\beta) = 
    \begin{cases}
        \text{genus of } [\alpha,\beta], &\text{if } S \text{ has finite genus}.
        \\
        0, &\text{if } S \text{ has infinite genus}.
    \end{cases}
\]
Then put $F_0(\alpha,\beta) = p_0(\alpha,\beta) + g_0(\alpha,\beta)$. Now remove the assumption that $\alpha$ and $\beta$ are disjoint. For separating curves $\gamma$ in $\alpha_+ \cap \beta_+$ or $\alpha_- \cap \beta_-$, define $p(\alpha,\beta;\gamma) = \abs{p_0(\alpha,\gamma) - p_0(\beta,\gamma)}$ and $g(\alpha,\beta;\gamma) = \abs{g_0(\alpha,\gamma) - g_0(\beta,\gamma)}$. The statements and proofs of Lemma \ref{lem:flux additivity} and Lemma \ref{lem:reference curve independence} all hold for $p_0$ and $g_0$. So define $p(\alpha,\beta) = p(\alpha,\beta;\gamma)$ and $g(\alpha,\beta) = g(\alpha,\beta;\gamma)$ for some (any) choice of separating curve $\gamma$ in $\alpha_+ \cap \beta_+$ or $\alpha_- \cap \beta_-$. The \textit{flux} between $\alpha$ and $\beta$, denoted by $F(\alpha,\beta)$, is then defined by
\[
    F(\alpha,\beta) = 
    p(\alpha,\beta) + g(\alpha,\beta).
\]
Again, whenever $\alpha$ and $\beta$ are disjoint we have $p(\alpha,\beta) = p_0(\alpha,\beta)$ and $g(\alpha,\beta) = g_0(\alpha,\beta)$, and therefore, $F(\alpha,\beta) = F_0(\alpha,\beta)$. 

\begin{lem}\label{lem:flux of neighbors is 1}
    For all separating curves $\alpha$ and $\beta$, if $d(\alpha,\beta) = 1$ then $F(\alpha,\beta)=1$.
\end{lem}

\begin{proof}
    By definition of the translatable curve graph $\Gamma$, $d(\alpha,\beta)=1$ means that $\alpha$ and $\beta$ are disjoint and $[\alpha,\beta]$ is homeomorphic to some $T \in \cS$. If $T=T_i$ for some $i=1,\dots,n$, then $T$ has exactly one discrete-type end and its genus is either 0 or infinite. So $p_0(\alpha,\beta) = 1$ and $g_0(\alpha,\beta) = 0$. If $T=T_{n+1}$, then $T$ has no discrete-type ends and its genus is 1. So $p_0(\alpha,\beta) = 0$ and $g_0(\alpha,\beta) = 1$. In either case, $F(\alpha,\beta) = F_0(\alpha,\beta) = 1$.
\end{proof}

Then the same arguments in the previous section may be used to show that $F$ is a pseudometric on $\Gamma$ which bounds $d$ from below. Also, it is still true in this setting that straight paths are geodesic. On the other hand, Lemma \ref{lem:flux=distance (flute)} does not immediately extend to the more general setting. Indeed, even if $\alpha$ and $\beta$ are disjoint, then it is possible that $d(\alpha,\beta) > F_0(\alpha,\beta)$. This is due to the subtle definition of $\Gamma$. For example, suppose $\Sigma = S^{\natural \bZ}$ and $S$ has a maximal end which is not of discrete type. If $[\alpha,\beta]$ has genus 1 and no ends, then $F_0(\alpha,\beta) = 1$, but $d(\alpha,\beta) \ne 1$ because $[\alpha,\beta]$ is not homeomorphic to $T_{N+1}$. However, we give a sufficient condition on $[\alpha,\beta]$ to guarantee that $d(\alpha,\beta) = F(\alpha,\beta)$.

\subsection{Full subsurfaces}

Recall from Section \ref{sec:translatable curve graph} the set $\cR= \s{f_1,\dots,f_N,c_1,\dots,c_M}$ of representatives of the equivalence classes of maximal ends of $S$ used to define the translatable curve graph of $\Sigma = S^{\natural \bZ}$.

\begin{defn}
    A subsurface $T \subset \Sigma$ is \textit{full} if $\End(T)$ contains an element of $E(f_i)$ for each $i=1,\dots,N$ and an element of $E(c_j)$ for each $j=1,\dots,M$. If $N=0$, then $S$ has finite positive genus and we further require $T$ to have positive genus.
\end{defn}

Note that the condition above is readily satisfied. For example, suppose $\alpha$ and $\beta$ are separating curves such that $\beta$ lies on the right side of $\alpha$. Even if $[\alpha,\beta]$ is not itself full, $[\alpha,h(\beta)]$ must be full because it contains a subsurface homeomorphic to $S$. The proof of the next lemma resembles the proof that $\Gamma$ is connected \cite[Lemma 4.7]{Schaffer-Cohen}.

\begin{lem}\label{lem:flux is equal to distance}
    Let $\alpha$ and $\beta$ be disjoint separating curves. If $[\alpha,\beta]$ is full, then there exists a straight path from $\alpha$ to $\beta$ of length $F(\alpha,\beta)$. In particular, $d(\alpha,\beta) = F(\alpha,\beta)$.
\end{lem}

\begin{proof}
    Assume without loss of generality that $\beta$ lies on the right side of $\alpha$. By \cite[Lemma 4.5]{Schaffer-Cohen}, the end space of $[\alpha,\beta]$ can be expressed as a disjoint union $\bigsqcup_{i=1}^k U_i$ where each $U_i$ is a stable neighborhood of a maximal end equivalent to one of the representatives in $\cR= \s{f_1,\dots,f_N,c_1,\dots,c_M}$ chosen above. Moreover, it follows from stability and maximality that each $U_i$ intersects $E(r)$ for exactly one $r \in R$. On the other hand, for each representative $r \in R$, at least one of the $U_i$ intersects $E(r)$ because $[\alpha,\beta]$ is full. In particular, $k \ge N+M$. Put $p = p(\alpha,\beta)$, and note that $N \le p \le k$. Denote by $\s{V_i}_{i=1}^p$ the set of $U_i$ which are stable neighborhoods of discrete-type ends. The remaining $U_i$ (if any exist) are stable neighborhoods of Cantor-type ends. Note that if $U_i$ and $U_{i'}$ both intersect some $E(c_j)$, then the disjoint union $U_i \sqcup U_{i'}$ is again a stable neighborhood which intersects $E(c_j)$. So the remaining $U_i$ may be combined to form a set $\s{W_i}_{i=1}^M$ where for each $i=1,\dots,M$, $W_i$ is a stable neighborhood of an end equivalent to $c_i$. In summary, we have
    \[
        \bigsqcup_{i=1}^k U_i = 
        \parens{\bigsqcup_{i=1}^\ell V_i} \sqcup \parens{\bigsqcup_{i=1}^M W_i}.
    \]
    Put $F = F_0(\alpha,\beta)$ and $g = g_0(\alpha,\beta)$. Since $[\alpha,\beta]$ is full, $F = p + g > 0$. For each $1 \le i \le M$, $E(c_i) \cap W_i$ is a Cantor set. So $W_i$ can be expressed as $\bigsqcup_{j=1}^F W_{i,j}$ where each $W_{i,j}$ is a stable neighborhood of an end in $E(c_i)$, and is thus homeomorphic to $W_i$. For $1 \le j \le p$, put $Z_j = V_j \sqcup \bigsqcup_{i=1}^M W_{i,j}$, and for $p+1\le j \le F$, put $Z_j = \bigsqcup_{i=1}^M W_{i,j}$. Then the end space of $[\alpha,\beta]$ is equal to $\bigsqcup_{j=1}^F Z_j$ and each $Z_j$ is homeomorphic to $\End(T)$ for some $T \in \cS$.

    Put $\alpha_0=\alpha$. For $j=1,\dots,p$, inductively define $\alpha_j$ to be the lasso curve $\alpha_{j-1}(\eta_j,\lambda_j)$, where the lasso $\eta_j \cup \lambda_j$ is contained in $[\alpha_{j-1},\beta]$, the end space of $C_{\eta_j}$ is equal to $Z_j$, and 
    \[
        \genus(C_{\eta_j}) =
        \begin{cases}
            0 \text{ or } \infty, \quad &1 \le j \le p, \\
            1, \quad &p < j \le F.
        \end{cases}
    \]
    By construction, for $1 \le j \le p$, $[\alpha_{j-1},\alpha]$ is homeomorphic to $T_i \in \cS$ for some $i=1,\dots,n$, and for $p < j \le F$, $[\alpha_{j-1},\alpha_j]$ is homeomorphic to $T_{n+1}$. Hence, $\alpha_{j-1}$ and $\alpha_j$ are neighbors in $\Gamma$ for all $j = 1,\dots,F$. Moreover, each $\alpha_j$ lies on the right side of $\alpha_{j-1}$ by construction, and $\alpha_F$ is isotopic to $\beta$ because $[\alpha_F,\beta]$ is an annulus. Therefore, $(\alpha_0=\alpha,\dots,\alpha_F=\beta)$ is a straight path from $\alpha$ to $\beta$ of length $F(\alpha,\beta)$. Since straight paths are geodesic, $d(\alpha,\beta) = F(\alpha,\beta)$.
\end{proof}

In fact, if $[\alpha,\beta]$ is full, then every geodesic in $\Gamma$ from $\alpha$ to $\beta$ must be a straight path.

\begin{lem}\label{lem:straight geodesics}
    Let $\alpha$ and $\beta$ be disjoint separating curves such that $\beta$ lies on the right side of $\alpha$ and $[\alpha,\beta]$ is full. If $(\gamma_0=\alpha,\gamma_1,\dots,\gamma_n=\beta)$ is any geodesic path in $\Gamma$ between $\alpha$ and $\beta$, then for each $i=1,\dots,n$, $\gamma_i$ lies on the right side of $\gamma_{i-1}$. In particular, $\gamma_i$ lies in  $[\alpha,\beta]$ for all $i=0,\dots,n$.
\end{lem}

\begin{proof}
    Suppose for contradiction that the conclusion fails, and let $j$ be the smallest index for which $\gamma_{j+1}$ lies on the left side of $\gamma_j$. Since $\gamma_i$ lies on the right side of $\gamma_{i-1}$ for each $i=1,\dots,j$, it follows from an inductive argument that $F(\alpha,\gamma_j) = j$. Let $\eta$ be a separating curve in $\alpha_-\cap{\gamma_{j+1}}_-$. Then
    \begin{align*}
        F_0(\eta,\beta) &= F_0(\eta,\alpha) + F_0(\alpha,\beta) \\
        F_0(\eta,\gamma_j) &= F_0(\eta,\alpha) + F_0(\alpha,\gamma_j) \\
        F_0(\eta,\gamma_j) &= F_0(\eta,\gamma_{j+1}) + F_0(\gamma_{j+1},\gamma_j).
    \end{align*}
    Lemma \ref{lem:flux is equal to distance} implies $F_0(\alpha,\beta) = d(\alpha,\beta) = n$, and Lemma \ref{lem:flux of neighbors is 1} implies $F_0(\gamma_{j+1},\gamma_j) = 1$. So
    \begin{align*}
        F(\beta,\gamma_{j+1}) &=
        \abs{F_0(\beta,\eta) - F_0(\gamma_{j+1},\eta)} \\&=
        \abs{F_0(\eta,\alpha) + F_0(\alpha,\beta) -F_0(\eta,\gamma_j) + F_0(\gamma_{j+1},\gamma_j)} \\&=
        \abs{F_0(\eta,\alpha) + n - (F_0(\eta,\alpha) + F_0(\alpha,\gamma_j)) + 1)} \\&=
        n-j+1.
    \end{align*}
    Hence, $d(\beta,\gamma_{j+1}) \ge F(\beta,\gamma_{j+1}) = n-j+1$. But recall that $(\gamma_0=\alpha,\gamma_1,\dots,\gamma_n=\beta)$ is geodesic path between $\alpha$ and $\beta$. In particular, $d(\beta,\gamma_{j+1}) = n-(j+1) = n-j-1$. So we have a contradiction.
\end{proof}

The next lemma is reminiscent of Lemma \ref{lem:lassoing a distant puncture moves away from the origin}, but crucially, it is not as powerful because it requires $\alpha$ and $\beta$ to be disjoint.

\begin{lem}
    Let $\alpha$ and $\beta$ be disjoint separating curves such that $[\alpha,\beta]$ is full. If $\alpha(\eta,\lambda)$ is a lasso curve adjacent to $\alpha$ in $\Gamma$, and $\alpha(\eta,\lambda)$ is not contained in $[\alpha,\beta]$, then $d(\alpha(\eta,\lambda),\beta) \ge d(\alpha,\beta)$.
\end{lem}

\begin{proof}
    We prove the contrapositive statement. Let $n = d(\alpha,\beta)$ and suppose $d(\alpha(\eta,\lambda),\beta) < n$. Since $\alpha(\eta,\lambda)$ and $\alpha$ are neighbors in $\Gamma$, we must have $d(\alpha(\eta,\lambda),\beta) = n-1$. Take a geodesic path of length $n-1$ between $\alpha(\eta,\lambda)$ and $\beta$, and include (the edge between $\alpha(\eta,\lambda)$ and) $\alpha$ and  to obtain a path of length $n$ between $\alpha$ and $\beta$. Then this is a geodesic path between $\alpha$ and $\beta$ which contains $\alpha(\eta,\lambda)$. By Lemma~\ref{lem:straight geodesics}, $\alpha(\eta,\lambda)$ is contained in $[\alpha,\beta]$.
\end{proof} 

Lastly, we define the notions of \textit{flux-right} and \textit{flux-left} in the exact same way as in the previous section, and the statements and proofs of Lemma \ref{lem:flux-right facts} all still hold in the more general setting.

\subsection{Hamming distance}\label{sec:hamming distance}

In the previous case of the bi-infinite flute, one of the tools we used was the operation of forgetting a puncture. However, it does not seem to be easy to extend this idea to the more general setting. How does one precisely define ``forgetting a discrete-type end" or ``forgetting a genus" as operations which map separating curves to separating curves? Thus, to compensate for the absence of this tool, we introduce another pseudometric on $\Gamma$ which provides a better lower bound on $d$ than flux does. For a subsurface $T \subset \Sigma$, denote by $\End_d(T)$ the subset of discrete-type ends in $\End(T)$. For each separating curve $\alpha$, put $P(\alpha) = \End_d(\alpha_+)$. 

\begin{defn}\label{defn:hamming-dist}
    The \textit{Hamming distance} between separating curves $\alpha$ and $\beta$, denoted by $H(\alpha,\beta)$, is defined by
    \[
        H(\alpha,\beta) = \abs{P(\alpha) \triangle P(\beta)} + g(\alpha,\beta).
    \]
\end{defn}

To get some intuition behind this definition, suppose that $S$ has zero or infinite genus, so that $g(\alpha,\beta)$ is always 0. Each separating curve $\alpha$ partitions the set of discrete-type ends of $\Sigma$, which is identified with $\bZ$, into two sets, thereby associating to $\alpha$ a bi-infinite binary sequence which stabilizes at 0 in one direction and at 1 in the other direction. Then $H(\alpha,\beta) = \abs{P(\alpha) \triangle P(\beta)}$ is equal to the classical Hamming distance between the two sequences associated to $\alpha$ and $\beta$.

\begin{lem}
    For all separating curves $\alpha$ and $\beta$, if $d(\alpha,\beta) = 1$ then $H(\alpha,\beta) = 1$.
\end{lem}

\begin{proof}
    By definition, $d(\alpha,\beta)=1$ means that $\alpha$ and $\beta$ are disjoint and $[\alpha,\beta]$ is homeomorphic to some $T \in \cS$. If $T=T_i$ for some $i=1,\dots,n$, then $T$ has exactly one discrete-type end $f$ and its genus is either 0 or infinite. Assuming $\alpha$ lies on the left side of $\beta$, this means that $P(\alpha) = P(\beta) \sqcup \s f$. So $P(\alpha) \triangle P(\beta) = \s f$ and $g_0(\alpha,\beta)=0$. If $T=T_{n+1}$, then $T$ has no discrete-type ends and its genus is 1. So $P(\alpha) = P(\beta)$ and $g_0(\alpha,\beta)=1$. In either case, $H(\alpha,\beta) = 1$.
\end{proof}

Moreover, $H$ is indeed a pseudometric. 

\begin{lem}
For all separating curves $\alpha,\beta,\gamma$ in $\Sigma$, we have $H(\alpha,\beta) \le H(\alpha,\gamma) + H(\gamma,\beta)$.
\end{lem}

\begin{proof}
It suffices to show that the inequality holds separately for $\abs{P(\alpha) \triangle P(\beta)}$ and $g(\alpha,\beta)$.
Indeed,
\[
    \abs{P(\alpha) \triangle P(\beta)} \le
    \abs{P(\alpha) \triangle P(\gamma) \cup P(\gamma) \triangle P(\beta)} \le
    \abs{P(\alpha) \triangle P(\gamma)} + \abs{P(\gamma) \triangle P(\beta)}
\]
and letting $\eta$ be a separating curve in $\alpha_+ \cap \beta_+ \cap \gamma_+$, 
\[
    g(\alpha,\beta) = 
    \abs{g_0(\alpha,\eta) - g_0(\beta,\eta)} =
    \abs{g_0(\alpha,\eta) - g_0(\gamma,\eta) + g_0(\gamma,\eta) - g_0(\beta,\eta)} \le
    g(\alpha,\gamma) + g(\gamma, \beta).
\]
\end{proof}

Furthermore, the Hamming distance gives a tighter lower bound on the graph distance $d$ in $\Gamma$. 

\begin{lem}\label{lem:Hamming distance is a lower bound}
For all separating curves $\alpha$ and $\beta$, we have $d(\alpha,\beta) \ge H(\alpha,\beta) \ge F(\alpha,\beta)$.
\end{lem}

\begin{proof}
The first inequality follows from the same argument used in the proof of Lemma \ref{lem:flux is a lower bound}. To get the second inequality, let $\gamma$ be a separating curve in $\alpha_+ \cap \beta_+$. By definition, $p_0(\alpha,\gamma) = \abs{\End_d([\alpha,\gamma])}$ and $p_0(\beta,\gamma) = \abs{\End_d([\beta,\gamma])}$. Since $\gamma_+$ is contained in both $\alpha_+$ and $\beta_+$, 
\[
    P(\alpha) \triangle P(\beta) = 
    \End_d([\alpha,\gamma]) \triangle \End_d([\beta,\gamma]).
\]
Hence,
\begin{align*}
    H(\alpha,\beta) &=
    \abs{P(\alpha) \triangle P(\beta)} + g(\alpha,\beta) \\&=
    \abs{\End_d([\alpha,\gamma]) \triangle \End_d([\beta,\gamma])} + g(\alpha,\beta) \\&\ge
    \abs{\abs{\End_d([\alpha,\gamma])} - \abs{\End_d([\beta,\gamma])}} + g(\alpha,\beta) \\&=
    \abs{p_0(\alpha,\gamma) - p_0(\beta,\gamma)} + g(\alpha,\beta) \\&=
    p(\alpha,\beta) + g(\alpha,\beta) \\&=
    F(\alpha,\beta).
\end{align*}
\end{proof}

We now aim to understand how lassoing affects the Hamming distance between curves.

\begin{lem}\label{lem:lassoing toggles discrete-type ends}
    Let $\alpha$ be a separating curve. If $\alpha(\eta,\lambda)$ is a lasso curve defined by $\alpha$ and some $\eta$ and $\lambda$, then $P(\alpha) \triangle P(\alpha(\eta,\lambda)) = \End_d(C_\eta)$.
\end{lem}

\begin{proof}
    First suppose $\eta$ lies on the right side of $\alpha$, so that $\End(C_\eta) \subset \End(\alpha_+)$. It follows from the construction of $\alpha(\eta,\lambda)$ that after lassoing, $\eta$ now lies on the left side of $\alpha(\eta,\lambda)$, and $\End(\alpha(\eta,\lambda)_+) = \End(\alpha_+) \setminus \End(C_\eta)$. Similarly, if $\eta$ lies on the left side of $\alpha$, then $\End(C_\eta)$ and $\End(\alpha_+)$ are disjoint, and $\End(\alpha(\eta,\lambda)_+) = \End(\alpha_+) \sqcup \End(C_\eta)$. Thus, in both cases, $\End(\alpha_+) \triangle \End(\alpha(\eta,\lambda)_+) = \End(C_\eta)$. Considering only the discrete-type ends in these sets then gives the result.
\end{proof}

We observe in the next lemma that given a pair of curves, applying the lassoing procedure to just one of the curves, with respect to a discrete-type end that is on the same side of both curves, increases the Hamming distance between the pair. 

\begin{lem} \label{lem:lassoing increases hamming distance}
Let $\alpha$ and $\beta$ be separating curves. Suppose $\alpha(\eta,\lambda)$ is a lasso curve such that $\eta$ lies on the same side of $\alpha$ and $\beta$, and $C_\eta$ has genus 0 if $S$ has finite genus. Then 
\[
    H(\alpha(\eta,\lambda),\beta) = 
    H(\alpha,\beta) + \abs{\End_d(C_\eta)}.
\]
\end{lem}

\begin{proof}
    If $S$ has infinite genus, then $g(\alpha(\eta,\lambda),\beta) = 0 = g(\alpha,\beta)$. On the other hand, if $S$ has finite genus, then $C_\eta$ has genus 0 by the hypothesis, and it follows that $g(\alpha(\eta,\lambda),\beta) = g(\alpha,\beta)$. 
    
    By the proof of Lemma \ref{lem:lassoing toggles discrete-type ends}, $P(\alpha(\eta,\lambda)) = P(\alpha) \setminus \End_d(C_\eta)$ if $\eta$ lies on the right side of $\alpha$, and $P(\alpha(\eta,\lambda)) = P(\alpha) \sqcup \End_d(C_\eta)$ if $\eta$ lies on the left side of $\alpha$. In either case, whether $\eta \subset \alpha_+ \cap \beta_+$ or $\eta \subset \alpha_- \cap \beta_-$,
    \[
        P(\alpha(\eta,\lambda)) \triangle P(\beta) = (P(\alpha) \triangle P(\beta)) \sqcup \End_d(C_\eta).
    \]
    Altogether, we get
    \begin{align*}
        H(\alpha(\eta,\lambda),\beta) &=
        \abs{P(\alpha(\eta,\lambda)) \triangle P(\beta)} + g(\alpha(\eta,\lambda),\beta) \\&=
        \abs{P(\alpha) \triangle P(\beta)} + \abs{\End_d(C_\eta)} + g(\alpha,\beta) \\&=
        H(\alpha,\beta) + \abs{\End_d(C_\eta)}.
    \end{align*}
\end{proof}

\subsection{Main theorem}

Now we are ready to prove the main theorem.

\begin{thm}
    If $\Sigma$ is a stable avenue surface with a discrete-type end, then $\Map(\Sigma)$ is one-ended.
\end{thm}

\begin{proof}
    By Theorem \ref{thm:mapping class group is quasi-isometric to its translatable curve graph}, it suffices to prove that $\Gamma$ is one-ended. Fix a vertex $o \in \Gamma$ and let $R > 0$ be any integer. To prove that $\Gamma$ is one-ended, it suffices by Lemma \ref{lem:one-ended criterion} to show that for any $\alpha,\beta \in \Gamma$ with $d(\alpha,o) = d(\beta,o) = 6R$, there exists a path from $\alpha$ to $\beta$ which is disjoint from $B = B(o,R)$.

    Like in the proof of Theorem \ref{thm:flute case}, the transitive action of $\Map(\Sigma)$ on $\Gamma$ allows us to assume that $\beta$ is the boundary curve of some copy of $S$. Then we may use the translation $h$ to push $\beta$ as far towards $e_+$ or $e_-$ as we want. In total, we may assume without loss of generality that $\beta$ lies on the right side of $\alpha$ such that $[\alpha,\beta]$ is full and $d(\beta,o) \ge 3R$. Then Lemma \ref{lem:flux is equal to distance} gives a straight path $(\gamma_0=\alpha,\gamma_1,\dots,\gamma_k=\beta)$ from $\alpha$ and $\beta$. Since $\Sigma$ has at least one class of discrete-type ends, there exists a lasso $(\eta,\lambda)$ such that $\eta \subset \beta_+ \cap o_+$, $C_\eta$ is homeomorphic to $\widehat T_i$ for some $1 \le i \le n$, and $\lambda$ intersects each $\gamma_i$ exactly once. Consider the lasso path $(\gamma_0(\eta,\lambda),\dots,\gamma_k(\eta,\lambda))$. Note that $\alpha$ and $\gamma_0(\eta,\lambda)$ are neighbors in $\Gamma$, and so are $\beta$ and $\gamma_k(\eta,\lambda)$. Also, for each $i=0,\dots,k$, Lemma \ref{lem:lassoing increases hamming distance} implies that $H(\gamma_i(\eta,\lambda),o) = H(\gamma_i,o) + 1$. Since the lasso path is again a straight path, this procedure may be iterated indefinitely, and each iteration produces a new lasso path whose vertices have a Hamming distance from $o$ increased by 1. Repeat the procedure for a total of $2R$ iterations and denote by $(\gamma_0',\dots,\gamma_k')$ the final path obtained after the last iteration. Then for each $i=0,\dots,k$,
    \[
        d(\gamma_i',o) \ge
        H(\gamma_i',o) =
        H(\gamma_i,o) + 2R >
        R.
    \]
    So $(\gamma_0',\dots,\gamma_k')$ is disjoint from $B$. Furthermore, the endpoints of the lasso paths constructed in each iteration yield a path $P$ from $\alpha$ to $\gamma_0'$ and a path $Q$ from $\beta$ to $\gamma_k'$. For each vertex $u$ among the first $R+1$ vertices along $P$, $d(u,o) \ge 2R$ because $d(\alpha,o),d(\beta,o) \ge 3R$. For each vertex $v$ among the remaining vertices of $P$,
    \[
        d(v,o) \ge H(v,o) \ge H(\alpha,o) + R+1 > R.
    \]
    So $P$, and similarly $Q$, are also disjoint from $B$. The concatenation of all three paths then is a path from $\alpha$ to $\beta$ which is disjoint from $B$.
\end{proof}

\bibliographystyle{alpha}
\bibliography{references.bib}

\end{document}